\documentclass[journal]{IEEEtran}
\usepackage{cite}
\usepackage{amsmath,amssymb,amsfonts}
\usepackage{algorithmic}
\usepackage{graphicx}
\usepackage{textcomp}
\usepackage{float}
\usepackage{graphicx}
\usepackage{subfigure}
\usepackage{epsfig}
\usepackage{xcolor}
\usepackage{amsthm}
\newtheorem{theorem}{Theorem}
\newtheorem{lemma}{Lemma}

\newtheorem{definition}{Definition}

\newtheorem{assumption}{Assumption}

\def\BibTeX{{\rm B\kern-.05em{\sc i\kern-.025em b}\kern-.08em
    T\kern-.1667em\lower.7ex\hbox{E}\kern-.125emX}}
\begin{document}
\title{Event-Triggered Discrete-Time \\ Multivariable Extremum Seeking Systems}

\author{
Victor Hugo Pereira Rodrigues$^{a}$,
Tiago Roux Oliveira$^{a}$,
Miroslav Krsti{\' c}$^{b}$,
Frank Allg{\" o}wer$^{c}$
\thanks{$^{a}$ Department of Electronics and Telecommunication Engineering,\\ State University of Rio de Janeiro (UERJ), Rio de Janeiro--RJ, Brazil \\ (\texttt{victor.rodrigues@uerj.br;  tiagoroux@uerj.br})}
\thanks{$^{b}$  Department of Mechanical and Aerospace  Engineering,\\ University of California at San Diego (UCSD), La Jolla--CA, USA \\ (\texttt{mkrstic@ucsd.edu})}
\thanks{$^{c}$  Institute for Systems Theory and Automatic Control, University of \\Stuttgart, Stuttgart, Germany (\texttt{allgower@ist.uni-stuttgart.de})}
}

\maketitle

\begin{abstract}
This paper introduces a discrete-time event-triggered extremum seeking framework for real-time optimization of multivariable nonlinear systems. In contrast to conventional discrete-time extremum seeking implementations that rely on periodic input updates, the proposed scheme updates the control action only when a state-dependent triggering condition is met, enabling aperiodic execution and substantial reduction of actuation and communication effort. The resulting closed-loop architecture reconciles two structurally different paradigms: the periodic excitation required for gradient recovery in extremum seeking and the aperiodic philosophy of event-triggered control. 
By combining discrete-time averaging arguments with Lyapunov-based analysis, we prove practical convergence of the system trajectories to a neighborhood of the unknown extremum and exponential stability of the corresponding average dynamics despite the loss of uniform sampling. The results show that appropriately designed triggering rules preserve the essential optimization mechanisms of classical extremum seeking while drastically reducing the number of input updates. Numerical simulations demonstrate that comparable optimization performance can be achieved with significantly fewer actuation events, highlighting the suitability of the method for resource-aware digital and networked control implementations. 
\end{abstract}

\begin{IEEEkeywords}
Adaptive Control; Extremum Seeking; Event-Triggered Control; Multiparameter Optimization; Discrete-Time Systems.
\end{IEEEkeywords}

\section{Introduction}
\label{sec:introduction}

\IEEEPARstart{I}{n} many modern control applications, the objective is no longer limited to stabilizing or regulating a system around a prescribed operating point, but rather to \textit{optimize performance in real time} in the presence of uncertainty, limited modeling knowledge, and resource constraints. This shift is particularly evident in networked and embedded control systems, where sensing, computation, and communication are no longer abundant, and where controllers must operate efficiently while adapting to unknown or time-varying performance landscapes. As a result, two major research directions have gained significant attention over the past decades: feedback optimization methods that enable online performance improvement without explicit models, and resource-aware control strategies that reduce the frequency of sensing and actuation without compromising stability. Among the most prominent representatives of these directions are \textit{extremum seeking} \cite{Aminde_2013,7035079, 7892977, 8827636, 10772652, 11219171} and \textit{event-triggered control} \cite{11316674, 11316266, 11270207, 11316250, 11250918}, whose theoretical developments have progressed largely in parallel within the control literature.

Extremum seeking (ES) \cite{KW:2000,TRoux:2022} is a feedback-based optimization methodology that drives a dynamical system toward the extremum of an unknown performance map without requiring an explicit model of the plant or objective function. Its theoretical foundations are rooted in averaging and singular perturbation analyses, which explain how periodic perturbations and demodulation mechanisms generate gradient-like information from measured outputs and enable provable convergence to neighborhoods of optimal operating points. While originally developed in continuous time, ES has been systematically extended to discrete-time implementations, multivariable settings, and scenarios involving time-varying or uncertain dynamics, where sampled-data analysis becomes essential to characterize stability and convergence properties. A defining feature of most ES schemes is their intrinsically \textit{time-triggered} nature: periodic dithers, synchronous demodulation, and fixed update rates are fundamental to the analytical machinery that justifies gradient recovery. This structural reliance on uniform sampling and periodic excitation makes ES particularly amenable to rigorous analysis, but also reveals limitations when one seeks to operate under resource constraints that preclude frequent or regular updates.

Conversely, event-triggered control (ETC) \cite{T:2007} departs from conventional periodic sampled-data implementations by enforcing control updates only when state- or output-dependent conditions are satisfied, leading to inherently \textit{aperiodic} execution. Developed within the frameworks of Lyapunov stability theory, hybrid systems, and sampled-data control, ETC provides systematic mechanisms to guarantee stability and performance while significantly reducing sensing, communication, and actuation effort. Triggering rules are typically designed to bound the discrepancy between continuously evolving signals and their last transmitted or updated values, ensuring that stability is preserved despite irregular update times. This paradigm has become central in the study of networked and resource-constrained control systems, where communication bandwidth, energy consumption, and computational load are critical concerns. However, ETC theory is traditionally developed for stabilization and regulation objectives, and its interaction with optimization-oriented adaptive schemes---particularly those relying on periodic excitation and synchronous updates---remains far less understood, especially in discrete-time multivariable contexts.

The structural contrast between the intrinsically time-triggered nature of extremum seeking and the inherently aperiodic philosophy of event-triggered control raises fundamental questions when both paradigms are combined within a single feedback loop. On the one hand, ES analysis relies critically on periodic excitation, synchronous demodulation, and fixed update rates to justify gradient recovery through averaging arguments; on the other hand, ETC intentionally disrupts periodicity by enforcing state-dependent, irregular updates to reduce resource usage. As a result, directly embedding ES schemes into an event-triggered architecture is far from straightforward, since the very mechanisms that guarantee convergence of ES may be altered by the loss of uniform sampling. This tension becomes even more pronounced in discrete-time and multivariable settings, where sampled-data effects, coupling among channels, and asynchronous updates interact in nontrivial ways.

Motivated by this gap, this paper develops a systematic event-triggered framework for discrete-time multivariable extremum seeking systems that preserves the optimization capabilities of ES while significantly reducing the number of control updates. Building upon discrete-time ES analysis and Lyapunov-based triggering design, we show how appropriate event conditions can be constructed so that gradient estimation and convergence properties are maintained despite aperiodic execution. The proposed results extend recent scalar formulations of our companion conference paper \cite{CDC26} to the multivariable case and provide rigorous guarantees that explicitly account for coupling effects and sampled-data dynamics, thereby establishing a theoretical bridge between two major paradigms of feedback optimization and resource-aware control.

It is worth noting that our previous works \cite{RODRIGUES2022555,RODRIGUES202310307,10886645,11312658,PEREIRARODRIGUES202513,VHPR:2023a} on event-triggered extremum seeking were developed in a continuous-time setting, where the triggering mechanism was designed in conjunction with the continuous evolution of the plant and adaptation dynamics. In contrast, the present work addresses the fundamentally different challenges that arise in discrete-time implementations, where sampling, input holding, and update constraints play a central role in the system behavior. Hence, the extension from continuous-time to discrete-time event-triggered extremum seeking is therefore not a straightforward discretization, but constitutes the main original contribution of this paper, requiring a distinct design and analysis framework tailored to discrete-time dynamics.

From a theoretical standpoint, the analysis combines discrete-time averaging \cite{BFS:1988,PPY:2004} with Lyapunov arguments to show practical convergence of the closed-loop trajectories to a neighborhood of the unknown extremum and exponential stability of the associated average dynamics. These results provide a rigorous justification that the introduction of the triggering mechanism does not compromise the fundamental ES behavior, but rather complements it with a resource-aware update policy.

Beyond the specific design proposed here, this work highlights a broader and still underexplored direction: event-triggered control in discrete time. While most of the event-triggered literature has been developed in continuous time \cite{EDK:2010}, discrete-time formulations are particularly natural for digital implementations, where sampling, computation, and communication constraints inherently shape how and when updates can occur \cite{frank:2023,frank:2023b}. In this sense, the proposed scheme is not only an ES contribution, but also a step toward bridging discrete-time control and event-triggered mechanisms in a practically meaningful way.

Simulation results illustrates that a substantial reduction in input updates can be achieved without sacrificing optimization performance, reinforcing the relevance of the approach for model-free resource-constrained real-time applications.

\section{Preliminaries}
\label{sec:preliminaries}

\textbf{Notation.} Throughout the manuscript, the 2-norm (Euclidean) of vectors and induced norm of matrices are denoted by double bars $\|\cdot\|$ while absolute value of scalar variables are denoted by single bars $|\cdot|$. The terms $\lambda_{\min}(\cdot)$ and $\lambda_{\max}(\cdot)$ denote the minimum and maximum eigenvalues of a given positive definite matrix, respectively. Consider $\varepsilon \in \lbrack -\varepsilon_{0}\,, \varepsilon_{0} \rbrack \subset \mathbb{R}$ and the mappings $\delta_{1}(\varepsilon)$ and $\delta_{2}(\varepsilon)$, where $\delta_{1}: \lbrack -\varepsilon_{0}\,, \varepsilon_{0} \rbrack \to \mathbb{R}$ and $\delta_{2}: \lbrack -\varepsilon_{0}\,, \varepsilon_{0} \rbrack \to \mathbb{R}$. The function $\delta_1(\varepsilon)$ has magnitude of order $\delta_2(\varepsilon)$, denoted by $\delta_{1}(\varepsilon) = \mathcal{O}(\delta_{2}(\varepsilon))$, if there exist positive constants $k$ and $c$ such that $|\delta_{1}(\varepsilon)| \leq k |\delta_{2}(\varepsilon)|$, for all $|\varepsilon|<c$.

\begin{definition}[Shift Operator, \cite{AW:1997}] \label{def:shiftOperator}

Let $f=\{f[k]\}_{k\in\mathbb{Z}}$ denote a discrete-time sequence, where $f[k] \in \mathbb{R}^{n}$ represents the value of the sequence at index $k \in \mathbb{Z}$. Assume that $f[k]$ is obtained by sampling a continuous-time signal $f_a(t) \in \mathbb{R}^{n}$ with sampling period $h>0$, {\it i.e.}, $f[k] = f_a(kh)$, $k\in\mathbb{Z}$. The {\bf forward shift operator} $q$ is defined by $qf[k] = f[k+1]$. Similarly, the {\bf backward shift operator} is $q^{-1}f[k] = f[k-1]$. In the $z$-domain, $q$ is identified with the complex variable $z \in \mathbb{C}$.
\end{definition}

\begin{figure*}[h!]
\centering
\includegraphics[width=17cm]{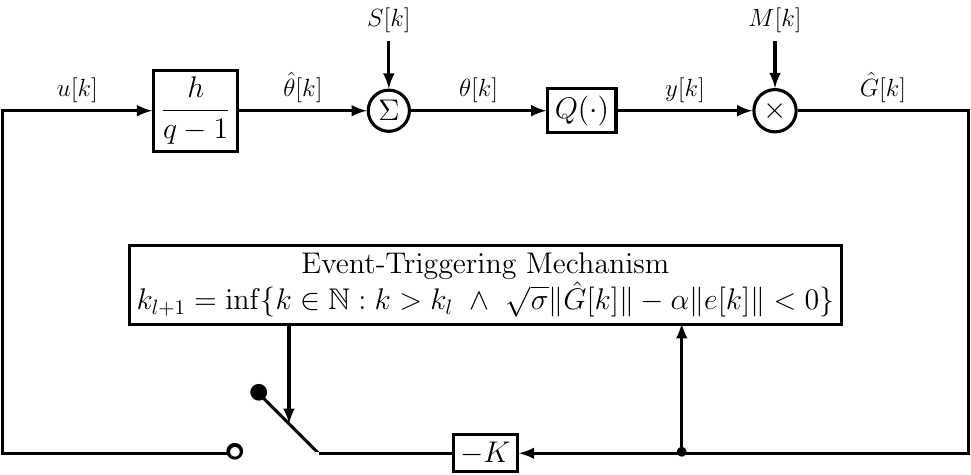}
\caption{Event-Triggered Discrete-Time Gradient-based Multivariable  Extremum Seeking.}
\label{fig:BD_GradientES_v2}
\end{figure*}

\section{Problem Formulation}
\label{sec:MainResults}

We consider the following  nonlinear static map
\begin{align}
y[k]&=Q(\theta[k])= Q^{\ast}+\frac{1}{2}(\theta[k]-\theta^{\ast})^{\top}H^{\ast}(\theta[k]-\theta^{\ast}) \label{eq:y_v2} \\
&=Q^{\ast}+\frac{1}{2}\sum_{l=1}^{N}\sum_{m=1}^{N}H_{lm}^{\ast}(\theta_{l}[k]-\theta^{\ast}_{l})(\theta_{m}[k]-\theta^{\ast}_{m})\,, \label{eq:y_v3}
\end{align}
where $Q^{\ast}\in \mathbb{R}$ is the extremum point, $H^{\ast}=H^{\ast \top} \in \mathbb{R}^{n \times n}$ is the Hessian matrix, $\theta^{\ast} \in \mathbb{R}^{n}$ is the optimizer, and $\theta[k]\in \mathbb{R}^{n}$ is the input map. The cost function $Q(\theta[k])$ is the map to be optimized and its parameters $H^{\ast}$ and $\theta^{\ast}$ in (\ref{eq:y_v2}) are not explicitly known; however, we have access to measurements of $y[k] \in \mathbb{R}$ and can adjust $\theta[k]$. 

If the Hessian matrix is positive definite ({\it i.e.}, $H^{\ast} > 0$), the map (\ref{eq:y_v2}) is convex and attains a minimum at $\theta = \theta^{\ast}$. Conversely, if $H^{\ast} < 0$, the function becomes concave and the extremum at $\theta^{\ast}$ corresponds to a maximum. Therefore, by evaluating the sign of $H^{\ast}$, one can directly infer whether the extremum seeking algorithm should be interpreted as solving a minimization or maximization problem.

Although (\ref{eq:y_v2}) is a polynomial in $\theta$, it cannot be identified from a finite number of input/output pairs $(\theta, y) $ due to the real-time nature of extremum seeking optimization. Furthermore, extremum seeking is a versatile approach that can handle any unknown analytic function $Q(\theta)$, which can be expressed as a Taylor series expansion around a point $\theta^{\ast}$, where the function attains a minimum or maximum. This assumption enables a local quadratic approximation of the nonlinear map $Q(\theta)$, serving as a foundational principle and justifying our focus on quadratic maps. Therefore, while the analysis was conducted for a quadratic map, our approach is not limited to quadratic functions $Q(\theta)$.

\subsection{Classical Discrete-Time Gradient-based Extremum Seeking}

The multivariable gradient-based Extremum Seeking (GradientES) approach for this multivariable static map is illustrated in Fig.~\ref{fig:BD_GradientES_v2}. In this scheme, the feedback or adaptation gain is
\begin{align}
K=\text{diag}\left\{K_{{1}}\,,K_{{2}}\,,\ldots\,,K_{{n}}\right\}\,, \label{eq:K}
\end{align}
while the dither signals and demodulations are defined as (see \cite{GKN:2012}):
\begin{align}
S[k]&= \left[a_{1}\sin\left[\omega_1h k\right],\ldots,a_{i}\sin\left[\omega_ih k\right],\ldots,a_{n}\sin\left[\omega_nh k\right]\right]^{\top}\!\!, \label{eq:S_v1} \\
M[k]&=2\left[\frac{\sin\left[\omega_1h k\right]}{a_{1}},\ldots,\frac{\sin\left[\omega_ih k\right]}{a_{i}},\ldots,\frac{\sin\left[\omega_nh k\right]}{a_{n}}\right]^{\top}, \label{eq:M_v1}
\end{align}
with nonzero amplitudes $a_{i}$. Moreover, the probing frequencies $\omega_{i}$'s can be selected as
\begin{align}
\omega_{i}=\omega_{i}'\omega \,, \quad i \in \left\{1,\ldots\,,n\right\}\,, \label{eq:omegai_event}
\end{align}
where $\omega$ is a positive constant and $\omega_{i}'$ is a rational number. In \cite{GKN:2012}, the following assumption is originally introduced.
\begin{assumption}\label{assumption_w}
The probing frequencies satisfy
\begin{align}
\omega'_{i} 	\notin \left\{\omega'_{j}\,,~\frac{1}{2}(\omega'_{j}+\omega'_{k})\,,~\omega'_{j}+2\omega'_{k}\,,~\omega'_{k}\pm\omega'_{l}\right\}\,, \label{eq:omega_iNotIn}
\end{align}
for all $i$, $j$, $k$ and $l$.
\end{assumption} 

Under this assumption, the probing frequencies are selected so that specific combinations of frequencies do not generate DC components that could bias the gradient estimate and lead to an incorrect direction during the demodulation and averaging process, as studied in \cite{GKN:2012}.

The scheme in Fig.~\ref{fig:BD_GradientES_v2} highlights the core components of the ES architecture: the gradient estimation mechanism driven by the periodic perturbation-demodulation technique and the adaptation block that adjusts the input vector using the estimated gradient in order to estimate $\theta^{\ast}$.

Moreover, if Assumption~\ref{assumption_w} holds, an accurate estimate of the $i$-th component of the unknown gradient vector can be obtained by means of $\hat{G}[k] = M[k]y[k] \in \mathbb{R}^{n}$, whose $i$-th component can be expressed as
\begin{align}
\hat{G}_{i}[k]&=\frac{2}{a_{i}}\sin\left[\omega_ih k\right]y[k]\,. \label{eq:Gi}
\end{align}
As mentioned above, the output of the integrator, the vector $\hat{\theta}[k]\in \mathbb{R}^{n}$, give us an estimate of $\theta^{\ast} \in \mathbb{R}^{n}$ such that the \textit{estimation error} is defined by: 
\begin{align}
\tilde{\theta}[k]&:=\hat{\theta}[k]-\theta^{\ast}\,.\label{eq:tildeThetai_v1}
\end{align}
Still referring to Fig.~\ref{fig:BD_GradientES_v2}, the input to the nonlinear map $Q(\theta)$ in (\ref{eq:y_v2}) is given by $\theta[k] = \hat{\theta}[k] + S[k]$. For analysis purposes, and using (\ref{eq:S_v1}) and (\ref{eq:tildeThetai_v1}), the $i$-th component of the input vector $\theta[k]$ can be expressed in terms of the $i$-th component of the unknown estimation error vector $\tilde{\theta}[k]$, and the $i$-th component of the unknown optimal parameter vector $\theta^{\ast}$, as well as the $i$-th component of the known dither signal $S[k]$, as follows
\begin{align}
\theta_{i}[k]&=\tilde{\theta}_{i}[k]+a_{i}\sin\left[\omega_ih k\right]+\theta_{i}^{\ast}\,.\label{eq:thetai_v1}
\end{align}
Therefore, by plugging (\ref{eq:y_v3}),  (\ref{eq:M_v1}) and (\ref{eq:thetai_v1}) into (\ref{eq:Gi}), the $i$-th component of the estimate of the unknown gradient vector can be expressed as
\begin{align}
\hat{G}_{i}[k]&=\frac{1}{a_{i}}\sin[\omega_{i}hk]\tilde{\theta}^{\top}[k]H^{\ast}\tilde{\theta}[k]\nonumber \\
&\quad+\sum_{j=1}^{n}H_{ij}^{\ast}[k]\tilde{\theta}_{j}[k] +\delta_{i}[k] \,, \label{eq:hatG_20250522_1} \\
H_{ij}^{\ast}[k]:&=H_{ij}^{\ast}+\Delta H_{ij}^{\ast}[k]\,, \label{eq:calligraH_ij_20250522_1} \\
\Delta H_{ij}^{\ast}[k]&:=-H_{ij}^{\ast}\cos[2\omega_{i}hk]\nonumber \\
&\quad+\sum_{\substack{l=1 \\ l\neq i}}^{n}H_{lj}^{\ast}\frac{a_{l}}{a_{i}}\cos[(\omega_{i}-\omega_{l})hk] \nonumber \\
&\quad-\sum_{\substack{l=1 \\ l\neq i}}^{n}H_{lj}^{\ast}\frac{a_{l}}{a_{i}}\cos[(\omega_{i}+\omega_{l})hk]\,, \label{eq:DeltaCalligraH_ij_20250522_1} \\
\delta_{i}[k]&:=\frac{2Q^{\ast}}{a_{i}}\sin[\omega_{i}hk] \nonumber \\
&\quad+\frac{1}{4}\sum_{j=1}^{n}\sum_{l=1}^{n}H_{jl}^{\ast}\frac{a_{j}a_{l}}{a_{i}}\sin[(\omega_{i}+\omega_{j}-\omega_{l})hk] \nonumber \\
&\quad-\frac{1}{4}\sum_{j=1}^{n}\sum_{l=1}^{n}H_{jl}^{\ast}\frac{a_{j}a_{l}}{a_{i}}\sin[(\omega_{i}-\omega_{j}+\omega_{l})hk] \nonumber \\
&\quad-\frac{1}{4}\sum_{j=1}^{n}\sum_{l=1}^{n}H_{jl}^{\ast}\frac{a_{j}a_{l}}{a_{i}}\sin[(\omega_{i}+\omega_{j}+\omega_{l})hk] \nonumber \\
&\quad-\frac{1}{4}\sum_{j=1}^{n}\sum_{l=1}^{n}H_{jl}^{\ast}\frac{a_{j}a_{l}}{a_{i}}\sin[(\omega_{i}-\omega_{j}-\omega_{l})hk]\,. \label{eq:Delta_i_20250222_1}
\end{align}

Hence, by using (\ref{eq:calligraH_ij_20250522_1})--(\ref{eq:Delta_i_20250222_1}), we can define the time-varying matrices $H^{\ast}[k] \in \mathbb{R}^{n \times n}$, $\Delta H^{\ast}[k] \in \mathbb{R}^{n \times n}$, and the time-varying vector $\delta[k] \in \mathbb{R}^{n }$, as follows
\begin{align}
H^{\ast}[k] &:= H^{\ast}+\Delta H^{\ast}[k]\,, \label{eq:calligraH}\\
\Delta H^{\ast}[k]&:= \begin{bmatrix}
													\Delta H^{\ast}_{11}[k] & \Delta H^{\ast}_{12}[k] & \ldots & \Delta H^{\ast}_{1n}[k] \\
													\Delta H^{\ast}_{21}[k] & \Delta H^{\ast}_{22}[k] & \ldots & \Delta H^{\ast}_{2n}[k] \\
													\vdots                         & \vdots                         & \ddots & \vdots                         \\
													\Delta H^{\ast}_{n1}[k] & \Delta H^{\ast}_{n2}[k] & \ldots & \Delta H^{\ast}_{nn}[k] \\
												 \end{bmatrix} \,, \label{eq:DeltaCalligraH} \\
						 \delta[k] &:= \begin{bmatrix}
													\delta_{1}[k] \,, 
													\delta_{2}[k] \,,
													\ldots\,,
													\delta_{n}[k]
												 \end{bmatrix}^{\top} \,. \label{eq:Delta}
\end{align}
Then, by using (\ref{eq:M_v1}) and (\ref{eq:calligraH})--(\ref{eq:Delta}), we can express (\ref{eq:hatG_20250522_1}) in the compact form
\begin{align}
\hat{G}[k]&\mathbb{=}\frac{1}{2}M[k]\tilde{\theta}^{\top}[k]H^{\ast}\tilde{\theta}[k]+\left(H^{\ast}+\Delta H^{\ast}[k]\right)\tilde{\theta}[k]+\delta[k], \label{eq:hatG_20240302_2}
\end{align}
where $\Delta H^{\ast}[k]$, defined in (\ref{eq:DeltaCalligraH}), and $\delta[k]$, given in (\ref{eq:Delta}), are time-varying matrix and vector, respectively, both with zero mean.

On the other hand, from the Event-Triggered Discrete-Time Gradient-based Multivariable Extremum Seeking scheme depicted in Fig.~\ref{fig:BD_GradientES_v2}, by using the shift operator given in Definition~\ref{def:shiftOperator}, the discrete estimate of $\theta^{\ast}$ can be found as 
\begin{align}
    \hat{\theta}[k]=\frac{h}{q-1}u[k]\,. \label{eq:thetaHat_20260313}
\end{align}
Then, by plugging (\ref{eq:thetaHat_20260313}) in (\ref{eq:tildeThetai_v1}), the dynamics that governs $\tilde{\theta}[k]$, is given by
\begin{align}
\tilde{\theta}[k]&=\frac{h}{q-1}u[k]-\theta^{\ast} \label{eq:dtildeThetadt_20250206_1}\,, 
\end{align}
or, equivalently, 
\begin{align}
\tilde{\theta}[k+1]&=\tilde{\theta}[k]+hu[k] \label{eq:tildeThetaK+1_20260313_1}\,, 
\end{align}
with $u[k]= [u_{1}[k]\,,u_{2}[k]\,,\ldots\,,u_{n}[k]]^{\top} \in \mathbb{R}^{n}$ once $\theta^{\ast}[k]=\theta^{\ast}$ such that $(q-1)\theta^{\ast}[k]=\theta^{\ast}[k+1]-\theta^{\ast}[k]=\theta^{\ast}-\theta^{\ast}=0$. Moreover, by using (\ref{eq:hatG_20240302_2}), the gradient estimate in the iteration $k+1$ is given by
\begin{align}
\hat{G}[k+1]&=\frac{1}{2}M[k+1]\tilde{\theta}^{\top}[k+1]H^{\ast}\tilde{\theta}[k+1] \nonumber \\
&\quad+\left(H^{\ast}+\Delta H^{\ast}[k+1]\right)\tilde{\theta}[k+1]+\delta[k+1] \,. \label{eq:hatG_20260313_2} 
\end{align}
Now, plugging (\ref{eq:tildeThetaK+1_20260313_1}) in (\ref{eq:hatG_20260313_2}), and \textcolor{blue}{adding} and \textcolor{red}{subtracting} the terms $\frac{1}{2}M[k]\tilde{\theta}^{\top}[k]H^{\ast}\tilde{\theta}[k]$, $\frac{h^{2}}{2}M[k]u^{\top}[k]H^{\ast}u[k]$, $\left(H^{\ast}+\Delta H^{\ast}[k]\right)\tilde{\theta}[k]$, $hM[k]\tilde{\theta}^{\top}[k]H^{\ast}u[k]$ and $\delta[k]$, equation (\ref{eq:hatG_20260313_2}) can be rewritten as 
\begin{align}
\hat{G}[k\mathbb{+}1]&=\textcolor{blue}{\frac{1}{2}M[k]\tilde{\theta}^{\top}[k]H^{\ast}\tilde{\theta}[k]+\left(H^{\ast}+\Delta H^{\ast}[k]\right)\tilde{\theta}[k]+\delta[k]} \nonumber \\
&\quad +\frac{1}{2}(M[k+1]\textcolor{red}{-M[k]})\tilde{\theta}^{\top}[k]H^{\ast}\tilde{\theta}[k]\nonumber \\
&\quad +\frac{h^{2}}{2}(\textcolor{blue}{M[k]}+M[k+1]\textcolor{red}{-M[k]})u^{\top}[k]H^{\ast}u[k]\nonumber \\
&\quad +h(\textcolor{blue}{M[k]}+M[k+1]\textcolor{red}{-M[k]})\tilde{\theta}^{\top}[k]H^{\ast}u[k]\nonumber \\
&\quad +\left(H^{\ast}+\Delta H^{\ast}[k+1]\textcolor{red}{-H^{\ast}-\Delta H^{\ast}[k]}\right)\tilde{\theta}[k] \nonumber \\
&\quad+h\left(H^{\ast}\textcolor{blue}{+\Delta H^{\ast}[k]}+\Delta H^{\ast}[k+1]\textcolor{red}{-\Delta H^{\ast}[k]}\right)u[k] \nonumber \\
&\quad+\delta[k+1]\textcolor{red}{-\delta[k]} \,. \label{eq:hatG_20260313_3} 
\end{align}
Then, defining $\Delta \tilde{H}^{\ast}[k]=\Delta H^{\ast}[k+1]-\Delta H^{\ast}[k]$, $\tilde{M}[k]=M[k+1]-M[k]$  and $\tilde{\delta}[k]=\delta[k+1]-\delta[k]$, and using (\ref{eq:hatG_20240302_2}), equation (\ref{eq:hatG_20260313_3}) becomes  
\begin{align}
\hat{G}[k+1]&=\hat{G}[k] +\frac{1}{2}\tilde{M}[k]\tilde{\theta}^{\top}[k]H^{\ast}\tilde{\theta}[k] \nonumber \\
&\quad +\frac{h^{2}}{2}(M[k]+\tilde{M}[k])u^{\top}[k]H^{\ast}u[k]\nonumber \\
&\quad +h(M[k]+\tilde{M}[k])\tilde{\theta}^{\top}[k]H^{\ast}u[k]+\tilde{M}[k]\tilde{\theta}[k]\nonumber \\
&\quad+h\left(H^{\ast}+\Delta H^{\ast}[k]+\Delta \tilde{H}^{\ast}[k]\right)u[k]+\tilde{\delta}[k]\,. \label{eq:hatG_20260313_4} 
\end{align}
From (\ref{eq:M_v1}), (\ref{eq:DeltaCalligraH_ij_20250522_1}) and (\ref{eq:Delta_i_20250222_1}), it is easy to verify that $M[k]$, $\Delta H^{\ast}[k]$ and $\delta[k]$ as well as $\tilde{M}[k]$, $\Delta \tilde{H}^{\ast}[k]$ and $\tilde{\delta}[k]$ have null mean values as well.

\subsection{Discrete-Time Gradient-Based Tuning Law}

We consider the following static gradient-feedback tuning law, given by 
\begin{align}
u[k]=-K\hat{G}[k] \,, \quad \forall k \in \mathbb{Z} \label{eq:u}
\end{align}
with the matrix gain $K$ being designed such that $I_{n}-hH^{\ast}K$ is Schur,  with $I_{n} \in \mathbb{R}^{n \times n}$ being an identity matrix. 

We assume that the control law (\ref{eq:u}) stabilizes the system at the corresponding average equilibrium $\hat{G}_{\rm{av}} \equiv 0$ with local exponential convergence. This means that any trajectory of $\hat{\theta}[k]$ starting sufficiently near to the extremum point $\theta^{\ast}$ will converge to a neighborhood of it exponentially, with uniform decay and overshoot bounds. It simply formalizes the idea that the control law is designed for local stabilization, regardless of the specific value of $\theta$. The control law (\ref{eq:u}) does not rely on detailed knowledge of the mapping (\ref{eq:y_v3}) or the system dynamics (\ref{eq:tildeThetaK+1_20260313_1}) and (\ref{eq:hatG_20260313_2}). Its design is model-independent in that sense, which is particularly advantageous in scenarios where precise models are unavailable or hard to obtain.

\subsection{Event-Triggered Discrete-Time Multivariable Extremum Seeking} \label{sec:ET-GradientES}

In our approach, the control updates are executed at a sequence of iteration step $k \in \mathbb{Z}$, which are not predetermined at uniform intervals, but instead generated by an event-triggering mechanism. This mechanism is designed to ensure the closed-loop system retains both stability and robustness. 

The event-triggering mechanism for discrete-time multivariable Gradient-based extremum seeking operates by monitoring the evolution of the gradient estimate $\hat{G}[k]$, and determining whether a control update is required based on a prescribed deviation criterion. In other words, the control input $u[k]$ is only updated at discrete iteration $\{k_{l}\}_{l\in \mathbb{N}}$, which are generated by a condition that compares the current $\hat{G}[k]$ to its most recently sampled value $\hat{G}[k_{l}]$. If the difference between these two quantities exceeds a predefined threshold, an event is triggered and the control law is re-evaluated. 

Assuming no delays in the sensor-to-controller and controller-to-actuator communication paths, the control input is updated only at the discrete triggering iteration step $\{k_{l}\}_{l\in \mathbb{N}}$ and held constant over the inter-execution interval. In particular, the control input computed at the last triggering instant $k_l$ is applied for all $k \in [k_l, k_{l+1})$. Thus, the tuning law is given by
\begin{align}
u[k] = -K\hat{G}[k_l]\,, \quad \forall k \in [k_l, k_{l+1})\,, \quad l \in \mathbb{N}. 
\label{eq:u_discrete_event}
\end{align}
Since the control updates occur only at the triggering instants, the state-dependent quantities evolve between updates while the control input remains constant. This motivates the definition of the measurement error over the inter-execution interval as
\begin{align}
e[k] := \hat{G}[k_l] - \hat{G}[k]\,, \quad \forall k \in [k_l, k_{l+1})\,, \quad l \in \mathbb{N}. 
\label{eq:e_discrete_event}
\end{align}

Therefore, by using (\ref{eq:hatG_20240302_2}) and (\ref{eq:e_discrete_event}), $\forall k \in [k_l, k_{l+1})$, $l \in \mathbb{N}$, the event-triggered discrete-time multivariable Gradient-based tuning law (\ref{eq:u_discrete_event}) can be rewritten as
\begin{align}
u[k] &= -K\hat{G}[k]-Ke[k]  \,, \label{eq:u_20250602_v2} \\
&=-\frac{1}{2}KM[k]\tilde{\theta}^{\top}[k]H^{\ast}\tilde{\theta}[k]-K\left(H^{\ast}+\Delta H^{\ast}[k]\right)\tilde{\theta}[k] \nonumber \\
&\quad-K\delta[k]-Ke[k] \,. \label{eq:u_20250602_v3}
\end{align}
Now, plugging (\ref{eq:u_20250602_v2}) into (\ref{eq:hatG_20260313_4}) and (\ref{eq:u_20250602_v3}) into (\ref{eq:tildeThetaK+1_20260313_1}), we arrive at the following Input-to-State Stable (ISS) \cite{K:2002} representations for the dynamics of $\hat{G}[k]$ and $\tilde{\theta}[k]$ with respect to the error vector $e[k]$ in (\ref{eq:e_discrete_event}), $\forall k \in [k_l, k_{l+1})$, $l \in \mathbb{N}$: 
\begin{align}
\hat{G}[k\mathbb{+}1]&=\hat{G}[k] +\frac{1}{2}\tilde{M}[k]\tilde{\theta}^{\top}[k]H^{\ast}\tilde{\theta}[k] \nonumber \\
&\quad +\frac{h^{2}}{2}(M[k]+\tilde{M}[k])\hat{G}^{\top}[k]K^{\top}H^{\ast}K\hat{G}[k]\nonumber \\
&\quad +\frac{h^{2}}{2}(M[k]+\tilde{M}[k])e^{\top}[k]K^{\top}H^{\ast}Ke[k]\nonumber \\
&\quad +h^{2}(M[k]+\tilde{M}[k])\hat{G}^{\top}[k]K^{\top}H^{\ast}Ke[k]\nonumber \\
&\quad -h(M[k]+\tilde{M}[k])\tilde{\theta}^{\top}[k]H^{\ast}K\hat{G}[k] \nonumber \\
&\quad -h(M[k]+\tilde{M}[k])\tilde{\theta}^{\top}[k]H^{\ast}Ke[k] +\tilde{M}[k]\tilde{\theta}[k]\nonumber \\
&\quad-h\left(H^{\ast}+\Delta H^{\ast}[k]+\Delta \tilde{H}^{\ast}[k]\right)K\hat{G}[k]\nonumber \\
&\quad-h\left(H^{\ast}+\Delta H^{\ast}[k]+\Delta \tilde{H}^{\ast}[k]\right)Ke[k]+\tilde{\delta}[k]\,, \label{eq:hatG_20260318_1} \\
\tilde{\theta}[k\mathbb{+}1]&=\tilde{\theta}[k]-\frac{h}{2}KM[k]\tilde{\theta}^{\top}[k]H^{\ast}\tilde{\theta}[k] \nonumber \\
&\quad-hK\left(H^{\ast}+\Delta H^{\ast}[k]\right)\tilde{\theta}[k]-hK\delta[k]-hKe[k]\,. \label{eq:tildeThetaK+1_20260318_1}
\end{align}

In the following sections, we present a static event-triggered mechanism for the nonautonomous and nonlinear time-varying closed-loop system described in discrete-time by (\ref{eq:hatG_20260318_1}) and (\ref{eq:tildeThetaK+1_20260318_1}). In this framework, data transmission occurs only when the norm of the measurement error vector (\ref{eq:e_discrete_event}) exceeds a predefined threshold. This triggering mechanism operates without continuous communication, thereby reducing the communication burden while preserving system stability.

\subsection{Event-Triggering Mechanism} 

Definition~\ref{def:staticEvent_discrete} illustrates how the small design parameter $\sigma \in (0,1)$ and the error signal $e[k]$—representing the deviation between the gradient estimate $\hat{G}[k]$ and its last transmitted value—are leveraged to construct a static event-triggered mechanism for multivariable Gradient-based extremum seeking in discrete-time. The resulting strategy triggers the recalculation of the tuning law, as defined in (\ref{eq:u_discrete_event}) only when needed. The closed-loop system preserves asymptotic stability and guaranteed convergence, as supported by the theoretical results in \cite{EDK:2010}.

\begin{definition}[\small{Event-Triggering Condition}] \label{def:staticEvent_discrete}
The event-triggered discrete-time multivariable Gradient-based extremum seeking controller with static event-triggered condition consists of two components:
\begin{enumerate}
	\item A sequence of increasing triggering iteration step $\mathcal{K}$ such that $\mathcal{K}=\{k_{0}, k_{1}, k_{2}, \ldots\}$ with $k_{0}=0$, generated according to:\\
		$\bullet$ If $\left\{k \!\in\! \mathbb{Z}:~ k \!>\! k_{l} ~ \wedge ~ \sqrt{\sigma}\|\hat{G}[k]\| \!-\! \alpha \|e[k]\| \!<\! 0 \right\} = \emptyset$,
		then the set of triggering instants is $\mathcal{K}=\{k_{0}, k_{1}, \ldots, k_{l}\}$. \\
		
		$\bullet$ If $\left\{k \!\in\! \mathbb{Z}:~ k \!>\! k_{l} ~ \wedge ~ \sqrt{\sigma}\|\hat{G}[k]\| \!-\! \alpha \|e[k]\| \!<\! 0 \right\} \neq \emptyset$, 
		the next triggering instant is given by
		\begin{align}
		k_{l+1} = \inf\left\{k \in \mathbb{Z}:~ k \mathbb{>} k_{l} ~ \mathbb{\wedge} ~ \sqrt{\sigma}\|\hat{G}[k]\| \mathbb{-} \alpha \|e[k]\| \mathbb{<} 0 \right\}.
		\label{eq:k_next_event}
		\end{align}	
	\item A feedback control action (\ref{eq:u_discrete_event}) updated at the triggering instants $k_l$ and held constant for all $k \in [k_l, k_{l+1})$.
\end{enumerate}
\end{definition}

\subsection{Average Closed-Loop System}

First, to represent the event-triggered update directly in terms of the discrete-time variable $j$, we define the piecewise-constant signal
\begin{align}
    \zeta[k]:=\hat{G}[k_{l(k)}],\quad l(k)=\max\{i:\,k_i\le k\}, \label{eq:zeta_1}
\end{align}
where $l(k)$ denotes the index of the most recent triggering instant. By construction, \(\zeta[k]\) is constant, for all $k\in[k_l,k_{l+1})$, {\it i.e.}, 
\begin{align}
    \zeta[k]=\hat{G}[k_l]=\zeta_l. \label{eq:zeta_2}
\end{align}
Consequently, the triggering error (\ref{eq:e_discrete_event}) can be written as
\begin{align}
   e[k]=\zeta[k]-\hat{G}[k], \label{eq:e_discrete_event_2}
\end{align}
and, defining the augmented state as follows
\begin{align}
x[k]=\begin{bmatrix} x_{1}^{\top}[k]\,, x_{2}^{\top}[k]\end{bmatrix}^{\top}:=\begin{bmatrix} \hat{G}^{\top}[k]\,, \tilde{\theta}^{\top}[k]\end{bmatrix}^{\top}\,, \label{eq:X}
\end{align}
the closed-loop system (\ref{eq:hatG_20260318_1})--(\ref{eq:tildeThetaK+1_20260318_1}) assume the form
\begin{align}
x[k+1]&=x[k]+h f(k,x,h)\,, \label{eq:x_k+1_event}
\end{align}
where
\begin{align}
f_{1}(k,x,h)&=\frac{1}{2h}\tilde{M}[k]x^{\top}_{2}[k]H^{\ast}x_{2}[k] \nonumber \\
&\quad +\frac{h}{2}(M[k]+\tilde{M}[k])x^{\top}_{1}[k]K^{\top}H^{\ast}Kx_{1}[k]\nonumber \\
&\quad +\frac{h}{2}(M[k]+\tilde{M}[k])e^{\top}[k]K^{\top}H^{\ast}Ke[k]\nonumber \\
&\quad +h(M[k]+\tilde{M}[k])x^{\top}_{1}[k]K^{\top}H^{\ast}Ke[k]\nonumber \\
&\quad -(M[k]+\tilde{M}[k])x^{\top}_{2}[k]H^{\ast}Kx_{1}[k] \nonumber \\
&\quad -(M[k]+\tilde{M}[k])x^{\top}_{2}[k]H^{\ast}Ke[k] +\frac{1}{h}\tilde{M}[k]x_{2}[k]\nonumber \\
&\quad-\left(H^{\ast}+\Delta H^{\ast}[k]+\Delta \tilde{H}^{\ast}[k]\right)Kx_{1}[k]\nonumber \\
&\quad-\left(H^{\ast}+\Delta H^{\ast}[k]+\Delta \tilde{H}^{\ast}[k]\right)Ke[k]+\frac{1}{h}\tilde{\delta}[k]\,, \label{eq:f1_20260321_1} \\
f_{2}(k,x,h)&=-\frac{1}{2}KM[k]\tilde{\theta}^{\top}[k]H^{\ast}\tilde{\theta}[k] \nonumber \\
&\quad-K\left(H^{\ast}+\Delta H^{\ast}[k]\right)\tilde{\theta}[k]-K\delta[k]-Ke[k]\,. \label{eq:f2_20260321_1}
\end{align}
In (\ref{eq:x_k+1_event})--(\ref{eq:f2_20260321_1}), the dependence on the event-triggering mechanism enters through the piecewise-constant signal $\zeta[k]$ defined in (\ref{eq:zeta_1})--(\ref{eq:zeta_2}). Since $\zeta[k]$ remains constant over each inter-event interval, the field $f(k,x,h)$ coincides on $[k_l,\,k_{l+1})$ with the frozen field $f(k,x,h;\zeta_l)$, obtained by holding $\zeta[k]\equiv\zeta_l$. The frozen field is $T$-periodic in $k$, with
\begin{equation}
T = 2\pi \times \mathrm{LCM}\left\{\frac{1}{\omega_i}\right\},
\quad \forall i \in \{1,2,\ldots,n\},
\label{eq:T}
\end{equation}
where $\mathrm{LCM}$ denotes the least common multiple. Hence, it is possible to define
\begin{equation}
\omega := \frac{2\pi}{T}.
\label{eq:omega_event_1}
\end{equation}
However, $f(k,x,h)$ is \emph{not} globally $T$-periodic in $k$ once the triggering mechanism is active: the switching instants $\{k_l\}$ are state-dependent and do not lie on the dither grid. Consequently, the discrete-time averaging results of \cite{BFS:1988,PPY:2004} do not apply to $f(k,x,h)$ directly over the infinite horizon, but rather on each frozen interval, with the resulting $\mathcal{O}(h)$ estimate shown to propagate uniformly across resets (see Appendix B for the complete argument).

The averaging approach allows one to characterize in what sense the behavior of a constructed average autonomous system approximates that of the original non-autonomous system (\ref{eq:x_k+1_event}). Intuitively, when the system evolves on a slower time scale than the excitation, its behavior is predominantly governed by the average effect of the excitation over one period.

By applying the discrete-time averaging technique to (\ref{eq:x_k+1_event}), we obtain the average system
\begin{align}
x^{\rm{av}}[k+1] &= x^{\rm{av}}[k] + h f^{\rm{av}}(x^{\rm{av}}), \label{eq:xav_k+1_event} \\
f^{\rm{av}}(x^{\rm{av}}) &= \lim_{h \to 0}\lim_{T \to \infty} \frac{1}{T} \sum_{k=s+1}^{s+T} f(k,x^{\rm{av}},h). \label{eq:fav_event}
\end{align}

Therefore, ``freezing'' the average states variables of $x_{1}[k]=\hat{G}[k]$ and $x_{2}[k]=\tilde{\theta}[k]$ in (\ref{eq:x_k+1_event})--(\ref{eq:f2_20260321_1}), (\ref{eq:hatG_20260318_1}), (\ref{eq:tildeThetaK+1_20260318_1}) and (\ref{eq:e_discrete_event}), one gets,  for all $k \in [k_l, k_{l+1})$, the corresponding average system
\begin{align}
\hat{G}_{\rm{av}}[k+1]&=(I_{n} -h H^{\ast}K)\hat{G}_{\rm{av}}[k]-h H^{\ast}Ke_{\rm{av}}[k]\,, \label{eq:hatG_av_k+1_1} \\
\tilde{\theta}_{\rm{av}}[k+1]&=(I_{n}-h KH^{\ast})\tilde{\theta}_{\rm{av}}[k]-h Ke_{\rm{av}}[k]\,, \label{eq:tildeTheta_av_k+1_1} \\
\hat{G}_{\rm{av}}[k]&=H^{\ast}\tilde{\theta}_{\rm{av}}[k]\,, \label{eq:hatG_av_k_1} \\
e_{\rm{av}}[k] &= \hat{G}_{\rm{av}}[k_l] - \hat{G}_{\rm{av}}[k]\,, \quad \forall k \in [k_l, k_{l+1})\,, \quad l \in \mathbb{N}. 
\label{eq:e_av_k_1}
\end{align}

Since $H^{\ast}$ is symmetric and definite, it is invertible. Therefore, $H^{\ast-1}(I_{n}-h H^{\ast}K)H^{\ast} = I_{n}-h KH^{\ast}$. Hence, $ I_{n}-h KH^{\ast}$ is similar to $ I_{n}-h H^{\ast} K$, and so they have exactly the same eigenvalues. Because a matrix is Schur if and only if all its eigenvalues lie strictly inside the open unit disk, it follows that $ I_{n}-h H^{\ast} K$ is Schur $\Longrightarrow$ $ I_{n}-h KH^{\ast}$ is also Schur. Therefore, the following assumption is additionally considered throughout the paper.

\begin{assumption} \label{assumption_lyapunovEq}
 The matrices $(I_{n}-h H^{\ast}K)$ and $(I_{n}-h KH^{\ast})$ are Schur such that, for any given $Q=Q^{\top}>0$, there exist symmetric positive definite matrices $P=P^{\top}>0$ and $P'=P'^{\top}>0$ satisfying the following Lyapunov equations for discrete-time systems
\begin{align}
(I_{n}-h KH^{\ast})^{\top}P(I_{n}-h KH^{\ast})-P&=-Q \,, \label{eq:LyapEq_1} \\
(I_{n}-h H^{\ast}K)^{\top}P'(I_{n}-h KH^{\ast})-P'&=-Q \,. \label{eq:LyapEq_2} 
\end{align}   
Furthermore, there exists a known positive constant $\alpha$ satisfying 
\begin{small}
   \begin{align}
    \alpha > \frac{h \sqrt{ 4\|(I_{n} \mathbb{-}h H^{\ast}K)^{\top} \!\! P H^{\ast}K\|^{2}\mathbb{+}2\lambda_{\min}\{Q\}\|K^{\top} H^{\ast \top} \!\! P H^{\ast}K\|}}{\lambda_{\min}\{Q\}} \,. \label{eq:alpha}
   \end{align} 
\end{small}

\end{assumption}

Now, we introduce {\bf Definition~\ref{def:staticEvent_discrete_av}} as an average version of {\bf Definition~\ref{def:staticEvent_discrete}}.

\begin{definition}[\small{Average Event-Triggering Condition}] \label{def:staticEvent_discrete_av}
The event-triggered discrete-time multivariable Gradient-based extremum seeking controller with static event-triggered condition and the small design parameter $\sigma \in (0,1)$ consists of two components:
\begin{enumerate}
	\item A sequence of increasing triggering iteration step $\mathcal{K}$ such that $\mathcal{K}=\{k_{0}, k_{1}, k_{2}, \ldots\}$ with $k_{0}=0$, generated according to: \\
		$\bullet$ If  $\left\{k \!\in\! \mathbb{Z}:k \!>\! k_{l} \mathbb{\wedge} \sqrt{\sigma}\|\hat{G}_{\rm{av}}[k]\| \!-\! \alpha \|e_{\rm{av}}[k]\| \!<\! 0 \right\} = \emptyset$,
		then the set of triggering instants is $\mathcal{K}=\{k_{0}, k_{1}, \ldots, k_{l}\}$.\\
		
		$\bullet$ If $\left\{k \!\in\! \mathbb{Z}: k \!>\! k_{l} \mathbb{\wedge}  \sqrt{\sigma}\|\hat{G}_{\rm{av}}[k]\| \!-\! \alpha \|e_{\rm{av}}[k]\| \!<\! 0 \right\} \neq \emptyset$,
		the next triggering instant is given by
		\begin{align}
		k_{l\mathbb{+}1} \mathbb{=} \inf\left\{k \mathbb{\in} \mathbb{Z}: k \!>\! k_{l}  \mathbb{\wedge}  \sqrt{\sigma}\|\hat{G}_{\rm{av}}[k]\| \mathbb{-} \alpha \|e_{\rm{av}}[k]\| \mathbb{<} 0 \right\}\!\!\,.
		\label{eq:k_next_event_av}
		\end{align}
	
	\item A feedback control action (\ref{eq:u_discrete_event}) updated at the triggering instants $k_l$ and held constant for all $k \in [k_l, k_{l+1})$,
    		\begin{align}
			u_{\rm{av}}[k]=-K\hat{G}_{\rm{av}}[k_{l}] \,. \label{eq:U_MD2}
		\end{align}
\end{enumerate}
\end{definition}

\section{Stability Analysis}\label{sec:ET-DT-gradient-ES_stability}

Theorem~\ref{thm:NETESC_2} states the local exponential practical stability of the extremum seeking control of Fig.~\ref{fig:BD_GradientES_v2} based on event-triggering execution mechanism is ensured.

\begin{theorem} \label{thm:NETESC_2}
Consider the closed-loop average dynamics of the gradient estimate (\ref{eq:hatG_av_k+1_1}), the average error vector (\ref{eq:e_av_k_1}), under  Assumptions \ref{assumption_w} and \ref{assumption_lyapunovEq}, and the average \textbf{static} event-triggering mechanism given by \textbf{Definition \ref{def:staticEvent_discrete_av}}. For $h>0$ sufficiently small, the equilibrium $(\hat{G}_{\rm{av}}\,,\tilde{\theta}_{\rm{av}})=(0\,,0)$ is locally exponentially stable such that the following inequalities can be obtained for the non-average signals:
\begin{small}
\begin{align}
&\|\theta[k]\mathbb{-}\theta^{\ast}\|  \mathbb{\leq}  \sqrt{\frac{\lambda_{\max}\{P\}}{\lambda_{\min}\{P\}}}\|H^{\ast}\|\|H^{\ast\mathbb{-}1}\|\left(\!\!1\mathbb{-} \frac{\lambda_{\min}\{Q\}}{\lambda_{\max}\{P\}}\frac{\left(1 \mathbb{-} \sigma\right)}{2}\!\!\right)^{\!\!\frac{k}{2}} \nonumber \\
 &\times\|\theta[0]-\theta^{\ast}\|+\mathcal{O}\left(a+h\right),  \label{eq:normTheta_thm1} \\
    &|y[k] - Q^{\ast}| \leq 2\frac{\lambda_{\max}\!\left\{(-1)^{\beta}H^{\ast}\right\}}{\lambda_{\min}\!\left\{(-1)^{\beta}H^{\ast}\right\}}\frac{\lambda_{\max}\{P\}}{\lambda_{\min}\{P\}} \|H^{\ast}\|^{2}\|H^{\ast\mathbb{-}1}\|^{2} \nonumber \\
 &\times\left(1-\frac{\lambda_{\min}\{Q\}}{\lambda_{\max}\{P\}}\frac{\left(1-\sigma\right)}{2}\right)^{k}|y[0]-Q^{\ast}|+\mathcal{O}\left(a^{2}+h^2\right)\,, \label{eq:normY_thm1} 
\end{align}
\end{small}
 $\!\!$where $\beta = 0$ if $H^{\ast} > 0$ and $\beta = 1$ if $H^{\ast} < 0$, such that $(-1)^{\beta} H^{\ast} > 0$ in both cases, and $a=\sqrt{\sum_{i=1}^{n}a_{i}^{2}}$, with $a_i$ in (\ref{eq:S_v1}) and (\ref{eq:M_v1}). \textcolor{black}{In addition, the inter-event time satisfies $k_{l+1}-k_{l}\geq k^{\ast}\geq 2$ for all $l$ such that consecutive-iteration (persistent) triggering is excluded.
 }
\end{theorem} 

\begin{proof} \textcolor{black}{The proof strategy follows the standard methodology used throughout the extremum-seeking literature, both in continuous and discrete time. The analysis is first carried out for the average dynamics, where exponential stability of the equilibrium is established through a Lyapunov argument. Subsequently, the conclusions are transferred to the original event-triggered system by invoking discrete-time averaging results, which guarantee $\mathcal{O}(h)$-closeness between the trajectories of the original and average systems for sufficiently small sampling periods.} 

\textcolor{black}{This is precisely the approach adopted in classical ES analyses and does not constitute a relaxation of the stability claim. In particular, the practical exponential convergence estimates (\ref{eq:normTheta_thm1})--(\ref{eq:normY_thm1}) are ultimately stated for the original system, while the average system serves as an analytical vehicle to derive these properties.}

The proof of the theorem is divided into two parts: practical asymptotic stability and \textcolor{black}{exclusion of persistent triggering.}

\begin{flushleft}
\underline{\it A. Practical Asymptotic Stability}
\end{flushleft}

Now, consider the following Lyapunov function candidate for the average system (\ref{eq:hatG_av_k+1_1}),
\begin{align}
    V[\hat{G}_{\rm{av}}[k]]=\hat{G}_{\rm{av}}^{\top}[k]P\hat{G}_{\rm{av}}[k]\,, \quad P=P^{\top}>0\,. \label{eq:lyapfunc}
\end{align}
Along the trajectories of the system, by using (\ref{eq:LyapEq_1}),
\begin{align}
    &\Delta V[\hat{G}_{\rm{av}}[k]]= V[\hat{G}_{\rm{av}}[k+1]]-V[\hat{G}_{\rm{av}}[k]] \label{eq:lyapfunc_0} \\
    &\mathbb{=}\hat{G}_{\rm{av}}^{\top}[k+1]P\hat{G}_{\rm{av}}[k+1]-\hat{G}_{\rm{av}}^{\top}[k]P\hat{G}_{\rm{av}}[k] \nonumber \\
    &\mathbb{=}\mathbb{-}\hat{G}^{\top}_{\rm{av}}[k]Q\hat{G}_{\rm{av}}[k] \mathbb{-}2h\hat{G}^{\top}_{\rm{av}}[k](I_{n} \mathbb{-}h H^{\ast}K)^{\top}P H^{\ast}Ke_{\rm{av}}[k] \nonumber \\
    &\quad+h^{2}e_{\rm{av}}^{\top}[k] K^{\top} H^{\ast \top} P H^{\ast}Ke_{\rm{av}}[k]\,. \label{eq:lyapfunc_1}
\end{align}
Then, once \mbox{$\hat{G}^{\top}_{\rm{av}}[k](I_{n} \mathbb{-}h H^{\ast}K)^{\top}P H^{\ast}Ke_{\rm{av}}[k]\mathbb{\leq} \|\hat{G}^{\top}_{\rm{av}}[k]\| \mathbb{\times}$} \mbox{$\|(I_{n} \mathbb{-}h H^{\ast}K)^{\top}P H^{\ast}K\|\|e_{\rm{av}}[k]\|$, $e_{\rm{av}}^{\top}[k] K^{\top} H^{\ast \top} P H^{\ast} K$} \mbox{$ \mathbb{\times}e_{\rm{av}}[k] \leq\|K^{\top} H^{\ast \top} P H^{\ast}K\| \|e_{\rm{av}}[k]\|^2$}, and, remembering that $Q=Q^{\top}>0$, from Assumption~\ref{assumption_lyapunovEq}, such that $\lambda_{\min}\{Q\}\|\hat{G}_{\rm{av}}[k]\|^2\leq\hat{G}^{\top}_{\rm{av}}[k]Q\hat{G}_{\rm{av}}[k]$, an upper bound for (\ref{eq:lyapfunc_1}) can be found as
\begin{align}
    \Delta V[\hat{G}_{\rm{av}}[k]] &\leq - \lambda_{\min}\{Q\}\|\hat{G}_{\rm{av}}[k]\|^2 \nonumber \\
    &+2h\|\hat{G}^{\top}_{\rm{av}}[k]\| \|(I_{n} \mathbb{-}h H^{\ast}K)^{\top}P H^{\ast}K\|\|e_{\rm{av}}[k]\| \nonumber \\
    &+h^{2}\|K^{\top} H^{\ast \top} P H^{\ast}K\| \|e_{\rm{av}}[k]\|^2\,. \label{eq:lyapfunc_2}
\end{align}
Now, applying the Peter-Paul inequality \cite{W:1971}, $cd\leq \frac{\varepsilon c^2}{2} +\frac{d^2}{2\varepsilon}$ for all $c,d,\varepsilon>0$, with $c=\|\hat{G}^{\top}_{\rm{av}}[k]\|$, $d=\|(I_{n} \mathbb{-}h H^{\ast}K)^{\top}P H^{\ast}K\|\|e_{\rm{av}}[k]\|$ and $\varepsilon=\frac{\lambda_{\min}\{Q\}}{2h}$, the inequality (\ref{eq:lyapfunc_2}) is upper bounded by
 \begin{align}
    &\Delta V[\hat{G}_{\rm{av}}[k]] \leq - \lambda_{\min}\{Q\}\|\hat{G}_{\rm{av}}[k]\|^2+\frac{\lambda_{\min}\{Q\}}{2}\|\hat{G}_{\rm{av}}[k]\|^2 \nonumber \\
    &+\frac{2h^{2} \|(I_{n} \mathbb{-}h H^{\ast}K)^{\top}P H^{\ast}K\|^{2}}{\lambda_{\min}\{Q\}}\|e_{\rm{av}}[k]\|^{2} \nonumber \\
    &+h^{2}\|K^{\top} H^{\ast \top} P H^{\ast}K\| \|e_{\rm{av}}[k]\|^2 \nonumber \\
    &=- \frac{\lambda_{\min}\{Q\}}{2}\|\hat{G}_{\rm{av}}[k]\|^2 \nonumber \\
    &\mathbb{+}\frac{h^{2}( 2\|(I_{n} \mathbb{-}h H^{\ast}K)^{\top} \!\! P H^{\ast}K\|^{2}\mathbb{+}\lambda_{\min}\{Q\}\|K^{\top} H^{\ast \top} \!\! P H^{\ast}K\|)}{\lambda_{\min}\{Q\}} \nonumber \\
    &\times \|e_{\rm{av}}[k]\|^{2} \nonumber \\
    &=- \frac{\lambda_{\min}\{Q\}}{2}\left(\|\hat{G}_{\rm{av}}[k]\|^2 \right. \nonumber \\
    &\mathbb{-}\frac{h^{2}( 4\|(I_{n} \mathbb{-}h H^{\ast}K)^{\top} \!\! P H^{\ast}K\|^{2}\mathbb{+}2\lambda_{\min}\{Q\}\|K^{\top} H^{\ast \top} \!\! P H^{\ast}K\|)}{\lambda_{\min}^{2}\{Q\}} \nonumber \\
    &\left.\times \|e_{\rm{av}}[k]\|^{2}\right) \,. \label{eq:lyapfunc_3}
\end{align}   
By using (\ref{eq:alpha}), inequality (\ref{eq:lyapfunc_3}) is upper bounded by
 \begin{align}
    \Delta V[\hat{G}_{\rm{av}}[k]] &\leq - \frac{\lambda_{\min}\{Q\}}{2}\left(\|\hat{G}_{\rm{av}}[k]\|^2 - \alpha^{2} \|e_{\rm{av}}[k]\|^{2}\right) \nonumber \\
    &=- \frac{\lambda_{\min}\{Q\}}{2}\left(\|\hat{G}_{\rm{av}}[k]\| + \alpha \|e_{\rm{av}}[k]\|\right) \nonumber \\
    &\quad\times \left(\|\hat{G}_{\rm{av}}[k]\| - \alpha \|e_{\rm{av}}[k]\|\right)\,. \label{eq:lyapfunc_4}
\end{align}
In the proposed event-triggering mechanism, the update law is (\ref{eq:U_MD2}) such that the vector $u_{\rm{av}}[k]$ is held constant between two consecutive events, and therefore
\begin{align}
\alpha\|e_{\rm{av}}[k]\|& \leq \sqrt{\sigma}\|\hat{G}_{\rm{av}}[k]\|\,, \quad \sigma \in (0,1)\,. \label{eq:eAv_upperBound}
\end{align}
Now, by using (\ref{eq:eAv_upperBound}), inequality (\ref{eq:lyapfunc_4}) is upper bounded as 
 \begin{align}
    \Delta V[\hat{G}_{\rm{av}}[k]] &\leq - \frac{\lambda_{\min}\{Q\}}{2}\left(\|\hat{G}_{\rm{av}}[k]\| + \sqrt{\sigma}\|\hat{G}_{\rm{av}}[k]\|\right) \nonumber \\
    &\quad\times \left(\|\hat{G}_{\rm{av}}[k]\| - \sqrt{\sigma}\|\hat{G}_{\rm{av}}[k]\|\right) \nonumber \\
    &=- \frac{\lambda_{\min}\{Q\}}{2}\left(1 - \sigma\right)\|\hat{G}_{\rm{av}}[k]\|^2\,. \label{eq:lyapfunc_5}
\end{align}
By using the Rayleigh-Ritz Inequality \cite{K:2002}, one gets 
\begin{align}
\lambda_{\min}\{P\}\|\hat{G}_{\rm{av}}[k]\|^{2}\leq V[\hat{G}_{\rm{av}}[k]] \leq \lambda_{\max}\{P\}\|\hat{G}_{\rm{av}}[k]\|^{2}\,, \label{eq:Rayleigh-Ritz_pf2}
\end{align}
 and the following upper bound for (\ref{eq:lyapfunc_5}), one obtains
 \begin{align}
    \Delta V[\hat{G}_{\rm{av}}[k]] &\leq - \frac{\lambda_{\min}\{Q\}}{\lambda_{\max}\{P\}}\frac{\left(1 - \sigma\right)}{2}V[\hat{G}_{\rm{av}}[k]]\,. \label{eq:lyapfunc_6}
\end{align}
From (\ref{eq:lyapfunc_0}), $V[\hat{G}_{\rm{av}}[k+1]]= V[\hat{G}_{\rm{av}}[k]]+\Delta V[\hat{G}_{\rm{av}}[k]]$, and by using (\ref{eq:lyapfunc_6}), one has 
 \begin{align}
 V[\hat{G}_{\rm{av}}[k+1]]&= V[\hat{G}_{\rm{av}}[k]]+\Delta V[\hat{G}_{\rm{av}}[k]] \nonumber \\
    &\leq \left(1- \frac{\lambda_{\min}\{Q\}}{\lambda_{\max}\{P\}}\frac{\left(1 - \sigma\right)}{2}\right)V[\hat{G}_{\rm{av}}[k]]\,. \label{eq:lyapfunc_6_cc}
\end{align}
The recursive structure of (\ref{eq:lyapfunc_6_cc}) allows us to derive an upper bound for the evolution of the Lyapunov function (\ref{eq:lyapfunc}) over successive intervals such that we can iteratively relate the Lyapunov function at the current iteration step $k$ to its value at earlier triggering times. This approach leads to
 \begin{align}
 V[\hat{G}_{\rm{av}}[1]]&\leq \left(1- \frac{\lambda_{\min}\{Q\}}{\lambda_{\max}\{P\}}\frac{\left(1 - \sigma\right)}{2}\right)V[\hat{G}_{\rm{av}}[0]]\,, \nonumber \\
 V[\hat{G}_{\rm{av}}[2]]&\leq \left(1- \frac{\lambda_{\min}\{Q\}}{\lambda_{\max}\{P\}}\frac{\left(1 - \sigma\right)}{2}\right)V[\hat{G}_{\rm{av}}[1]] \nonumber \\
 &\leq \left(1- \frac{\lambda_{\min}\{Q\}}{\lambda_{\max}\{P\}}\frac{\left(1 - \sigma\right)}{2}\right)^{2}V[\hat{G}_{\rm{av}}[0]]\,, \nonumber \\
 \vdots ~~~~~ &\quad ~~~~~~~~~~~~~~~~~~~~~~~~\vdots \nonumber \\
 V[\hat{G}_{\rm{av}}[k]]&\leq \left(1- \frac{\lambda_{\min}\{Q\}}{\lambda_{\max}\{P\}}\frac{\left(1 - \sigma\right)}{2}\right)^{k}V[\hat{G}_{\rm{av}}[0]]\,.\label{eq:lyapfunc_7}
\end{align}
Now, lower bounding the left-hand side and upper bounding the right-hand size of (\ref{eq:lyapfunc_7}) with the corresponding sides of (\ref{eq:Rayleigh-Ritz_pf2}), one gets
 \begin{align}
 \lambda_{\min}\{P\}\|\hat{G}_{\rm{av}}[k]\|^{2}&\leq \left(1- \frac{\lambda_{\min}\{Q\}}{\lambda_{\max}\{P\}}\frac{\left(1 - \sigma\right)}{2}\right)^{k} \nonumber \\
 &\quad\times\lambda_{\max}\{P\}\|\hat{G}_{\rm{av}}[0]\|^{2}\,.\label{eq:lyapfunc_8_before} 
\end{align}
Therefore,
\begin{align}
 \|\hat{G}_{\rm{av}}[k]\|&\leq \sqrt{\frac{\lambda_{\max}\{P\}}{\lambda_{\min}\{P\}}}\left(1\mathbb{-} \frac{\lambda_{\min}\{Q\}}{\lambda_{\max}\{P\}}\frac{\left(1 \mathbb{-} \sigma\right)}{2}\right)^{\frac{k}{2}}\|\hat{G}_{\rm{av}}[0]\|\,.\label{eq:lyapfunc_8} 
\end{align}
Moreover, $H^\ast$ is invertible and, from (\ref{eq:hatG_av_k_1}), $\tilde{\theta}_{\rm{av}}[k] = H^{\ast -1}\,\hat{G}_{\rm{av}}[k]$,
the following inequality is established
\begin{align}
    \|H^{\ast -1}\|^{-1}\|\tilde{\theta}_{\rm{av}}[k]\| \leq \|\hat{G}_{\rm{av}}[k]\| \leq \|H^{\ast}\|\|\tilde{\theta}_{\rm{av}}[k]\|\,, \label{eq:normThetaG}
\end{align}
and, therefore, lower bounding the left-hand side and upper bounding the right-hand side of (\ref{eq:lyapfunc_8}) with the corresponding sides of (\ref{eq:normThetaG}), we can state that,
\begin{align}
 \|\tilde{\theta}_{\rm{av}}[k]\|&\mathbb{\leq} \sqrt{\frac{\lambda_{\max}\{P\}}{\lambda_{\min}\{P\}}}\|H^{\ast}\|\|H^{\ast\mathbb{-}1}\|\left(\!\!1\mathbb{-} \frac{\lambda_{\min}\{Q\}}{\lambda_{\max}\{P\}}\frac{\left(1 \mathbb{-} \sigma\right)}{2}\!\!\right)^{\frac{k}{2}} \nonumber \\
 &\quad\times\|\tilde{\theta}_{\rm{av}}[0]\|\,.\label{eq:lyapfunc_9} 
\end{align}
From (\ref{eq:LyapEq_1}), it follows that $P-Q=(I_{n}-h KH^{\ast})^{\top}P(I_{n}-h KH^{\ast})$. Since $P=P^{\top}>0$ and $Q=Q^{\top}>0$, there exists a matrix $R$ with independent columns such that $P=R^{\top}R$. Consequently, for any $v[k]\in \mathbb{R}^{n}$, one has $v^{\top}[k](I_{n}-h KH^{\ast})^{\top}P(I_{n}-h KH^{\ast})v[k]= v^{\top}[k](I_{n}-h KH^{\ast})^{\top}R^{\top}R(I_{n}-h KH^{\ast})v[k]=\|R(I_{n}-h KH^{\ast})v[k]\|^{2} \geq 0$. Therefore, the matrix $(I_{n}-h KH^{\ast})^{\top}P(I_{n}-h KH^{\ast}) \geq 0$, and thus $P-Q\geq0$. Let $\lambda_{1}\{P\}\leq\ldots\leq\lambda_{n}\{P\}$ and $\lambda_{1}\{Q\}\leq\ldots\leq\lambda_{n}\{Q\}$ denote the ordered eigenvalues of $P$ and $Q$, respectively. Since $P\geq Q$, it follows from the Loewner order \cite[Corollary~7.7.4]{HJ:2013} that $\lambda_{i}\{P\}\geq\lambda_{i}\{Q\}$ for each $i\in\{1,\ldots,n\}$. In particular, $\lambda_{\max}\{P\}\geq\lambda_{\max}\{Q\}\geq \lambda_{\min}\{Q\}$, which implies that $\frac{\lambda_{\min}\{Q\}}{\lambda_{\max}\{P\}}\leq1$. Moreover, since $\sigma \in (0,1)$, it follows that $\frac{\lambda_{\min}\{Q\}}{\lambda_{\max}\{P\}}\frac{\left(1 - \sigma\right)}{2} \in (0,1)$. Hence, the sequence $V[\hat{G}_{\rm{av}}[k]]$, $\hat{G}_{\rm{av}}[k]$, as well as $\tilde{\theta}_{\rm{av}}[k]$ in (\ref{eq:lyapfunc_7}), (\ref{eq:lyapfunc_8}) and (\ref{eq:lyapfunc_9}), converge exponentially to the origin.

Since (\ref{eq:f2_20260321_1}) is characterized by a $T$-periodic discrete-time vector field on each frozen inter-event interval, and noting that the corresponding average system with state variable $\tilde\theta_{\rm{av}}[k]$ is exponentially stable as established in (70), the conditions of \cite[Theorem~2.2.1]{BFS:1988} are satisfied on each such interval, with infinite-horizon validity established rigorously in Appendix~B (Theorem~4). %
In particular, for sufficiently small $h > 0$ and initial conditions $\tilde{\theta}[0]$ sufficiently close to the origin, the solutions of the original system and the average system remain close. 
Therefore,
\begin{align}
\|\tilde{\theta}[k]-\tilde{\theta}_{\rm{av}}[k]\|\leq\mathcal{O}\left(h\right)\,. \label{eq:plotnikov}
\end{align}
Now, adding and subtracting $\tilde{\theta}_{\rm{av}}[k]$ in the right-hand side of vector form of (\ref{eq:thetai_v1}), one has
\begin{align}
\theta[k] - \theta^\ast = \tilde{\theta}_{\rm{av}}[k]+\tilde{\theta}[k]-\tilde{\theta}_{\rm{av}}[k] + S[k]\,, \label{eq:q_ETSSC_v3cu}
\end{align}  
whose norm can be upper bounded by using the triangle inequality \cite{A:1957}, one gets
\begin{align}
\|\theta[k] - \theta^\ast\| \leq \|\tilde{\theta}_{\rm{av}}[k]\|+\|\tilde{\theta}[k]-\tilde{\theta}_{\rm{av}}[k]\| + \|S[k]\|\,.\label{eq:q_ETSSC_v3cu2}
\end{align} 
Since sine functions in (\ref{eq:S_v1}) are uniformly bounded, it is possible to derive a uniform upper bound for the Euclidean norm of $S[k]$ such that
\begin{align}
    \|S[k]\|\leq \sqrt{\sum_{i=1}^{n}a_{i}^{2}}=a \,.\label{eq:S_v2}
\end{align}
Thus, by using (\ref{eq:lyapfunc_9}), (\ref{eq:plotnikov}) and (\ref{eq:S_v2}), inequality (\ref{eq:q_ETSSC_v3cu2}) is upper bounded by 
\begin{align}
&\|\theta[k]\mathbb{-}\theta^{\ast}\|  \mathbb{\leq}  \sqrt{\frac{\lambda_{\max}\{P\}}{\lambda_{\min}\{P\}}}\|H^{\ast}\|\|H^{\ast\mathbb{-}1}\|\left(\!\!1\mathbb{-} \frac{\lambda_{\min}\{Q\}}{\lambda_{\max}\{P\}}\frac{\left(1 \mathbb{-} \sigma\right)}{2}\!\!\right)^{\!\!\frac{k}{2}} \nonumber \\
 &\times\|\theta[0]-\theta^{\ast}\|+\mathcal{O}\left(a+h\right)\,,\label{eq:q_ETSSC_v4}
\end{align} 
then, inequality (\ref{eq:normTheta_thm1}) is verified.

Notice that the Hessian matrix $H^{\ast} \in \mathbb{R}^{n \times n}$ is symmetric and definite, {\it i.e.}, either $H^{\ast}>0$ or $H^{\ast}<0$. Define
$\beta=
\begin{cases}
0, & \text{if } H^{\ast}>0\\
1, & \text{if } H^{\ast}<0 
\end{cases}\,,$
so that $(-1)^{\beta}H^{\ast}>0$. Then,
\begin{align}
\lambda_{\min}\!\left\{(\mathbb{-}1)^{\beta}H^{\ast}\right\}\|\theta[k]\mathbb{-}\theta^{\ast}\|^{2}
& \mathbb{\leq}(\theta[k]\mathbb{-}\theta^{\ast})^{\!\! \top \!\!}(\mathbb{-}1)^{\beta}H^{\ast}(\theta[k]\mathbb{-}\theta^{\ast})\nonumber \\
& \mathbb{\leq}\lambda_{\max}\!\left\{(\mathbb{-}1)^{\beta}H^{\ast}\right\}\|\theta[k]\mathbb{-}\theta^{\ast}\|^{2}. \label{eq:beta}
\end{align}
Now, from (\ref{eq:y_v2}), (\ref{eq:q_ETSSC_v4}) and (\ref{eq:beta}), one has 
\begin{align}
    |y[k] - Q^{\ast}|&=\frac{1}{2}|(\theta[k]-\theta^{\ast})^{\top}H^{\ast}(\theta[k]-\theta^{\ast})| \nonumber \\
    & \leq \frac{1}{2}\lambda_{\max}\!\left\{(-1)^{\beta}H^{\ast}\right\}\frac{\lambda_{\max}\{P\}}{\lambda_{\min}\{P\}} \|H^{\ast}\|^{2}\|H^{\ast\mathbb{-}1}\|^{2} \nonumber \\
 &~~~\times\left(1-\frac{\lambda_{\min}\{Q\}}{\lambda_{\max}\{P\}}\frac{\left(1-\sigma\right)}{2}\right)^{k}\|\theta[0]-\theta^{\ast}\|^{2} \nonumber \\
 &~~~+\lambda_{\max}\!\left\{(-1)^{\beta}H^{\ast}\right\}\sqrt{\frac{\lambda_{\max}\{P\}}{\lambda_{\min}\{P\}}} \|H^{\ast}\|\|H^{\ast\mathbb{-}1}\| \nonumber \\
 &~~~\times\left(1-\frac{\lambda_{\min}\{Q\}}{\lambda_{\max}\{P\}}\frac{\left(1-\sigma\right)}{2}\right)^{\frac{k}{2}}\|\theta[0]-\theta^{\ast}\| \nonumber \\
 &~~~\times\mathcal{O}\left(a+h\right)+\mathcal{O}\left(a+h\right)^2\,. \label{eq:yQ_1}
\end{align}
Applying the Peter-Paul inequality \cite{W:1971}, $cd\leq \frac{\varepsilon c^2}{2} +\frac{d^2}{2\varepsilon}$ for all $c,d,\varepsilon>0$, with $c=\sqrt{\dfrac{\lambda_{\max}\{P\}}{\lambda_{\min}\{P\}}} \|H^{\ast}\|\|H^{\ast\mathbb{-}1}\|\left(1-\frac{\lambda_{\min}\{Q\}}{\lambda_{\max}\{P\}}\frac{\left(1-\sigma\right)}{2}\right)^{\frac{k}{2}}\|\theta[0]-\theta^{\ast}\|$, $d=\mathcal{O}\left(a+h\right)$ and $\varepsilon=\frac{1}{2}$, the inequality (\ref{eq:yQ_1}) is upper bounded by
\begin{align}
    |y[k] - Q^{\ast}|& \leq \lambda_{\max}\!\left\{(-1)^{\beta}H^{\ast}\right\}\frac{\lambda_{\max}\{P\}}{\lambda_{\min}\{P\}} \|H^{\ast}\|^{2}\|H^{\ast\mathbb{-}1}\|^{2} \nonumber \\
 &~~~\times\left(1-\frac{\lambda_{\min}\{Q\}}{\lambda_{\max}\{P\}}\frac{\left(1-\sigma\right)}{2}\right)^{k}\|\theta[0]-\theta^{\ast}\|^{2} \nonumber \\
 &~~~+(2\lambda_{\max}\!\left\{(-1)^{\beta}H^{\ast}\right\}+1)\mathcal{O}\left(a+h\right)^2\,. \label{eq:yQ_2}
\end{align}
From (\ref{eq:beta}), $\lambda_{\min}\left\{(-1)^{\beta}H^{\ast}\right\}\|\theta[0]-\theta^{\ast}\|^{2} \leq |(\theta[0]-\theta^{\ast})^{ \top }H^{\ast}(\theta[0]\mathbb{-}\theta^{\ast})|=2|y[0]-Q^{\ast}|$ and, according to \cite[Definition~10.1]{K:2002}, $(2\lambda_{\max}\!\left\{(-1)^{\beta}H^{\ast}\right\}+1)\mathcal{O}\left(a+h\right)^{2}\leq \mathcal{O}\left(a+h\right)^{2} \leq \mathcal{O}\left(a+h\right)^{2} $ remains as an order of magnitude $\mathcal{O}\left(a^{2}+h^2\right)$. Thus, inequality (\ref{eq:yQ_2}) is upper bounded by
\begin{align}
    |y[k] - Q^{\ast}|& \leq 2\frac{\lambda_{\max}\!\left\{(-1)^{\beta}H^{\ast}\right\}}{\lambda_{\min}\!\left\{(-1)^{\beta}H^{\ast}\right\}}\frac{\lambda_{\max}\{P\}}{\lambda_{\min}\{P\}} \|H^{\ast}\|^{2}\|H^{\ast\mathbb{-}1}\|^{2} \nonumber \\
 &~~~\times\left(1-\frac{\lambda_{\min}\{Q\}}{\lambda_{\max}\{P\}}\frac{\left(1-\sigma\right)}{2}\right)^{k}|y[0]-Q^{\ast}| \nonumber \\
 &~~~+\mathcal{O}\left(a^{2}+h^2\right)\,. \label{eq:yQ_3}
\end{align}
such that inequality (\ref{eq:normY_thm1}) is verified.

\begin{flushleft}
\underline{\it \textcolor{black}{B. Exclusion of Persistent Triggering}}
\end{flushleft}

We now investigate the existence of a minimum interval between two consecutive triggering instants in the considered discrete-time framework. Due to the discrete-time nature of the system, inter-execution times are inherently lower bounded by one sampling step, and thus Zeno behavior is excluded by construction.

Accordingly, the objective is not to establish Zeno avoidance, but rather to ensure that the triggering mechanism yields a well-posed and nontrivial execution pattern. In particular, it is required that the event-triggering condition does not lead to pathological behaviors such as persistent triggering at every iteration, which would eliminate any computational or communication advantages.

To this end, we derive sufficient conditions under which the triggering rule guarantees strictly non-consecutive event occurrences over time, thereby ensuring a meaningful reduction in the number of control updates while preserving the desired closed-loop properties.

Notice, from (\ref{eq:lyapfunc_6}), for all $k \in [k_l, k_{l+1})$, by iteration,
 \begin{align}
 V[\hat{G}_{\rm{av}}[k]]&\leq \left(1- \frac{\lambda_{\min}\{Q\}}{\lambda_{\max}\{P\}}\frac{\left(1 - \sigma\right)}{2}\right)^{k-k_{l}}V[\hat{G}_{\rm{av}}[k_{l}]]\,,\label{eq:lyapfunc_10}
\end{align}
such that lower bounding the left-hand side and upper bounding the right-hand size of (\ref{eq:lyapfunc_10}) with the corresponding sides of (\ref{eq:Rayleigh-Ritz_pf2}), one gets
\begin{align}
 \|\hat{G}_{\rm{av}}[k]\|&\mathbb{\leq} \sqrt{\frac{\lambda_{\max}\{P\}}{\lambda_{\min}\{P\}}}\left(1\mathbb{-} \frac{\lambda_{\min}\{Q\}}{\lambda_{\max}\{P\}}\frac{\left(1 \mathbb{-} \sigma\right)}{2}\right)^{\!\!\!\frac{k\mathbb{-}k_{l}}{2}} \!\!\! \|\hat{G}_{\rm{av}}[k_{l}]\|\,.\label{eq:lyapfunc_11} 
\end{align}

Since (\ref{eq:f1_20260321_1}) is also characterized by a $T$-periodic discrete-time vector field on each frozen inter-event interval, and noting that the corresponding average system with state variable $\hat{G}_{\rm{av}}[k]$ is exponentially stable as established in (\ref{eq:lyapfunc_11}), the conditions of \cite[Theorem~2.2.1]{BFS:1988} are satisfied again on each such interval, with the infinite-horizon closeness estimate justified in Appendix~B (Theorem~4). %
In particular, for sufficiently small $h > 0$ and initial conditions $\hat{G}[0]$ sufficiently close to the origin, the solutions of the original system and the average system remain close. %
Therefore,
\begin{align}
\|\hat{G}[k]-\hat{G}_{\rm{av}}[k]\|\leq\mathcal{O}\left(h\right)\,. \label{eq:sastry_2}
\end{align}
Then, $\hat{G}[k]=\hat{G}_{\rm{av}}[k]+\hat{G}[k]-\hat{G}_{\rm{av}}[k]$ whose norm can be upper bounded by using the triangle inequality \cite{A:1957}, such that 
\begin{align}
\|\hat{G}[k]\|&=\|\hat{G}_{\rm{av}}[k]+\hat{G}[k]-\hat{G}_{\rm{av}}[k]\| \nonumber \\
&\leq \|\hat{G}_{\rm{av}}[k]\|+\|\hat{G}[k]-\hat{G}_{\rm{av}}[k]\| \,. \label{eq:normG_k_1}
\end{align}  
Thus, by using (\ref{eq:lyapfunc_11}), and (\ref{eq:sastry_2}), inequality (\ref{eq:normG_k_1}) is upper bounded by 
\begin{align}
\|\hat{G}[k]\|&\mathbb{\leq} \sqrt{\frac{\lambda_{\max}\{P\}}{\lambda_{\min}\{P\}}}\left(1\mathbb{-} \frac{\lambda_{\min}\{Q\}}{\lambda_{\max}\{P\}}\frac{\left(1 \mathbb{-} \sigma\right)}{2}\right)^{\!\!\!\frac{k\mathbb{-}k_{l}}{2}} \!\!\! \|\hat{G}_{\rm{av}}[k_{l}]\| \nonumber \\
&\quad+\mathcal{O}\left(h\right)\,. \label{eq:normG_k_2}
\end{align}  

Now, from (\ref{eq:hatG_av_k+1_1}) and (\ref{eq:e_av_k_1}), we can write, for all $k \in [k_l, k_{l+1})$, 
\begin{align}
\hat{G}_{\rm{av}}[k+1]&=(I_{n} -h H^{\ast}K)\hat{G}_{\rm{av}}[k]-h H^{\ast}Ke_{\rm{av}}[k]\,, \label{eq:hatG_av_k+1_2} \\
&=I_{n}\hat{G}_{\rm{av}}[k]-h H^{\ast}K\hat{G}_{\rm{av}}[k_l] \,, \label{eq:hatG_av_k+1_3} \\
e_{\rm{av}}[k+1] &= \hat{G}_{\rm{av}}[k_l] - \hat{G}_{\rm{av}}[k+1]\,. 
\label{eq:e_av_k_2}
\end{align}

By using (\ref{eq:hatG_av_k+1_3}), equation (\ref{eq:e_av_k_2}) can be rewritten as
\begin{align}
e_{\rm{av}}[k+1] &= \hat G_{\rm{av}}[k_l] - \hat G_{\rm{av}}[k+1] \nonumber\\
&= \big(\hat G_{\rm{av}}[k_l]-\hat G_{\rm{av}}[k]\big) + hH^*K\hat G_{\rm{av}}[k_l] \nonumber\\
&= e_{\rm{av}}[k] + hH^*K\hat G_{\rm{av}}[k_l]\,. \label{eq:eav_recursion}
\end{align}

Now, the iteration of (\ref{eq:hatG_av_k+1_3}) and (\ref{eq:eav_recursion}) for $k=\{k_l,k_l+1,k_l+2,\ldots\}$, with $k\in[k_l,k_{l+1})$, together with the condition $e_{\rm{av}}[k_l]=0$, yields
\begin{align}
\hat G_{\rm{av}}[k_l+1] &= (I_n-hH^*K)\hat G_{\rm{av}}[k_l]\,,\nonumber\\
&\;\;\vdots \nonumber\\
\hat G_{\rm{av}}[k] &= \big(I_n-(k-k_l)hH^*K\big)\hat G_{\rm{av}}[k_l]\,,\label{eq:Gav_closed}\\[4pt]
e_{\rm{av}}[k_l+1] &= hH^*K\hat G_{\rm{av}}[k_l]\,,\nonumber\\
&\;\;\vdots \nonumber\\
e_{\rm{av}}[k] &= (k-k_l)\,hH^*K\hat G_{\rm{av}}[k_l]\,. \label{eq:eav_closed}
\end{align}

By construction, (\ref{eq:Gav_closed})--(\ref{eq:eav_closed}) satisfy the identity $e_{\rm{av}}[k]=\hat G_{\rm{av}}[k_l]-\hat G_{\rm{av}}[k]$ exactly, as required by (\ref{eq:e_av_k_1}); the error accumulates linearly over the held interval, with neither a $(I_n-hH^*K)$ contraction factor nor a spurious factor of two.

By combining an upper bound on $\|e[k]\|$ with a lower bound on $\|\hat G[k]\|$, we can determine when the event-triggering condition must be violated. Therefore, by invoking the average method for discrete-time systems \cite[Theorem~2.2.1]{BFS:1988}, we have
\begin{align}
\|\hat G[k] - \hat G_{\rm{av}}[k]\| &\leq \mathcal{O}(h)\,, \label{eq:Gcloseness}\\
\|e[k] - e_{\rm{av}}[k]\| &\leq \mathcal{O}(h)\,, \label{eq:ecloseness}
\end{align}
such that, by using the triangle inequality [31], one additionally writes
\begin{align}
\|\hat G[k]\| &= \|\hat G_{\rm{av}}[k] + \hat G[k]-\hat G_{\rm{av}}[k]\| \nonumber\\
&\geq \big|\|\hat G_{\rm{av}}[k]\| - \|\hat G[k]-\hat G_{\rm{av}}[k]\|\big|\,, \label{eq:Gtriangle}\\
\|e[k]\| &= \|e_{\rm{av}}[k]+e[k]-e_{\rm{av}}[k]\| \nonumber\\
&\leq \|e_{\rm{av}}[k]\| + \|e[k]-e_{\rm{av}}[k]\|\,. \label{eq:etriangle}
\end{align}

Using (\ref{eq:Gav_closed}), (\ref{eq:eav_closed}), (\ref{eq:Gcloseness}) and (\ref{eq:ecloseness}), the following bounds are obtained for (\ref{eq:Gtriangle}) and (\ref{eq:etriangle}), respectively:
\begin{align}
\|\hat G[k]\| &\geq \Big\|\big(I_n-(k-k_l)hH^*K\big)\hat G_{\rm{av}}[k_l]\Big\| - \mathcal{O}(h) =: \underline{\hat G}[k]\,, \label{eq:Glow}\\
\|e[k]\| &\leq (k-k_l)\,h\,\|H^*K\hat G_{\rm{av}}[k_l]\| + \mathcal{O}(h) =: \bar e[k]\,. \label{eq:eup}
\end{align}

At this stage, the idea is to combine the upper bound on the error $\bar e[k]$ in (\ref{eq:eup}) with the lower bound on the gradient estimate $\underline{\hat G}[k]$ in (\ref{eq:Glow}) in order to determine the minimum number of iterations after which the event-triggering condition is necessarily violated. In other words, if at some iteration $k$ one has $\sqrt\sigma\,\underline{\hat G}[k] - \alpha\,\bar e[k] \leq 0$, then the condition (\ref{eq:k_next_event}) can no longer be ensured, and thus an event must occur. Therefore, the first iteration $k^*$ such that $\sqrt\sigma\,\underline{\hat G}[k]-\alpha\,\bar e[k]\leq0$ provides an estimate of the minimum number of iterations after which the event-triggering condition is necessarily violated, i.e.,
\begin{align}
k^* := \arg\min_{k\in\mathbb{N}}\Big\{&\alpha\,(k-k_l)\,h\,\|H^*K\hat G_{\rm{av}}[k_l]\| + \mathcal{O}(h) \nonumber\\
&\!\!\!\!\!\!\!\!\!\!\!\!\!\!\!\!\!\!\!\!\!\!\!\! \geq \sqrt\sigma\Big(\big\|\big(I_n-(k-k_l)hH^*K\big)\hat G_{\rm{av}}[k_l]\big\| - \mathcal{O}(h)\Big)\Big\}\,. \label{eq:kstar}
\end{align}

It remains to verify that $k^* \geq 2$ whenever $\hat G_{\rm{av}}[k_l] \neq 0$ and $h$ is sufficiently small. Evaluate the margin $\sqrt{\sigma}\,\bar{\hat G}[k] - \alpha\,\bar e[k]$ at $k = k_l + m$ for $m \in \{0,1\}$, using (\ref{eq:Glow})--(\ref{eq:eup}):

\emph{Step $m=0$:} $\bar e[k_l] = \mathcal{O}(h)$ and $\bar{\hat G}[k_l] = \|\hat G_{\rm{av}}[k_l]\| - \mathcal{O}(h)$, so
\begin{equation}
\sqrt{\sigma}\,\bar{\hat G}[k_l] - \alpha\,\bar e[k_l] = \sqrt{\sigma}\,\|\hat G_{\rm{av}}[k_l]\| - \mathcal{O}(h) > 0 \label{eq:margin_m0}
\end{equation}
for $h$ small enough, so no event fires at $k=k_l$.

\emph{Step $m=1$:} $\bar e[k_l+1] = h\|H^*K\hat G_{\rm{av}}[k_l]\| + \mathcal{O}(h) = \mathcal{O}(h)$ and $\bar{\hat G}[k_l+1] = \|(I_n-hH^*K)\hat G_{\rm{av}}[k_l]\| - \mathcal{O}(h) = \|\hat G_{\rm{av}}[k_l]\| + \mathcal{O}(h)$, so
\begin{equation}
\sqrt{\sigma}\,\bar{\hat G}[k_l+1] - \alpha\,\bar e[k_l+1] = \sqrt{\sigma}\,\|\hat G_{\rm{av}}[k_l]\| + \mathcal{O}(h) > 0\,, \label{eq:margin_m1}
\end{equation}
so no event fires at $k=k_l+1$ either. Hence $k^* \geq 2$, and combined with the lower-bound property in (97),
\begin{equation}
2 \leq k^* \leq k_{l+1} - k_l\,, \label{eq:kstar_final}
\end{equation}
so, by  \cite[Proposition~2]{EDK:2010}, the event-triggered rule (\ref{eq:k_next_event}) is nontrivial, in the sense that it takes at least two steps for the next controller update. This completes the proof.
\end{proof}

\section{Simulation results} \label{sec:sim_siso}

To illustrate the main features of the proposed discrete-time event-triggered extremum-seeking strategy, consider the nonlinear map in (\ref{eq:y_v2}), with input $\theta[k]\in \mathbb{R}^{2}$, output $y[k]\in \mathbb{R}$, and unknown parameters $H^{\ast}=\begin{bmatrix}
    100 & 30 \\
    30 & 20
\end{bmatrix}$, $Q^{\ast}=100$, and $\theta^{\ast}=\begin{bmatrix}
    2 &
    4
\end{bmatrix}^{\top}$. The dither signal is selected with parameters $a_{1}=a_{2}=0.1$, $\omega_{1}=7$ [rad/sec] and $\omega_{2}=5$ [rad/sec], while the event-triggering parameters are chosen as $\sigma=0.8$ and $\alpha=1.70$. The control gain is set to $K=\begin{bmatrix}
    0.05 & 0 \\
    0 & 0.05
\end{bmatrix}$, the step size is $h=0.18$ [sec], and the initial condition is $\hat{\theta}[0]=\begin{bmatrix}
    2.5 & 5
\end{bmatrix}^{\top}$.

In Fig.~\ref{fig:theta_siso} and Fig.~\ref{fig:y_siso}, we can check $\theta[k]$ converging to $\theta^{\ast}$ and the output $y[k]$ approaching its corresponding extremum value. 
Fig.~\ref{fig:U_siso} shows the aperiodic pattern of the control update instants, characterized by the nonuniform distribution of inter-execution intervals inherent to the event-triggered implementation. In this framework, the control input is updated only when the triggering condition is satisfied. Over 6000 iterations, the control law is updated only 32 times, yielding an average inter-execution interval of $33.75$ seconds.
Fig.\ref{fig:hatG_siso} presents the evolution of the gradient estimate, showing that it converges to zero while remaining well behaved along the iterations. The event-triggered mechanism ensures that all signals remain bounded and exhibit smooth convergence, as expected.

These results highlight the effectiveness of the proposed discrete-time event-triggered extremum-seeking scheme in significantly reducing the number of control updates while still ensuring convergence to the extremum.

\begin{figure*}[h!]
	\centering
	\subfigure[Input of the nonlinear map, \mbox{$\theta[k]$}. \label{fig:theta_siso}]{\includegraphics[width=8.8cm]{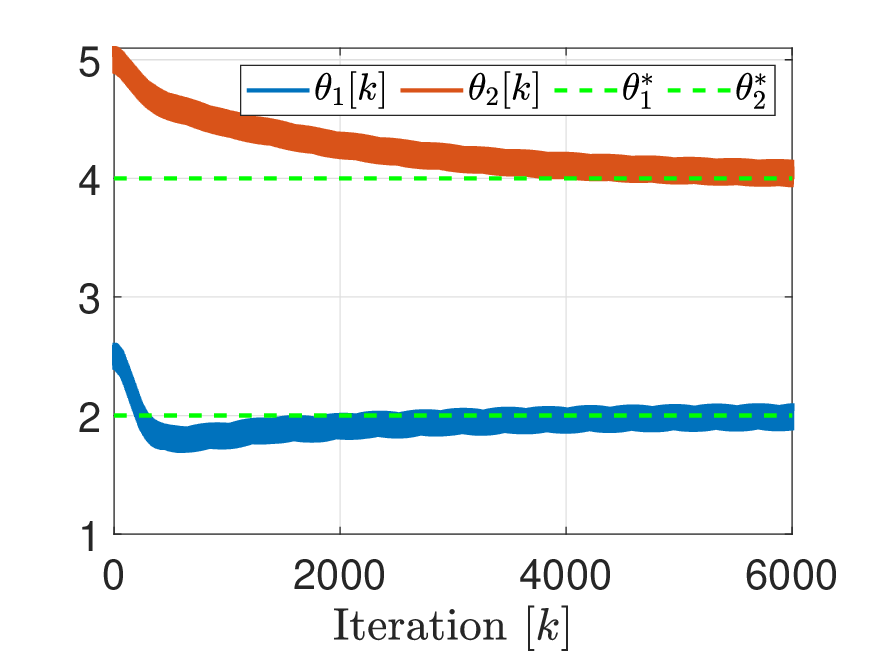}}
    ~~
    \subfigure[Output of the nonlinear map, \mbox{$y[k]$}. \label{fig:y_siso}]{\includegraphics[width=8.8cm]{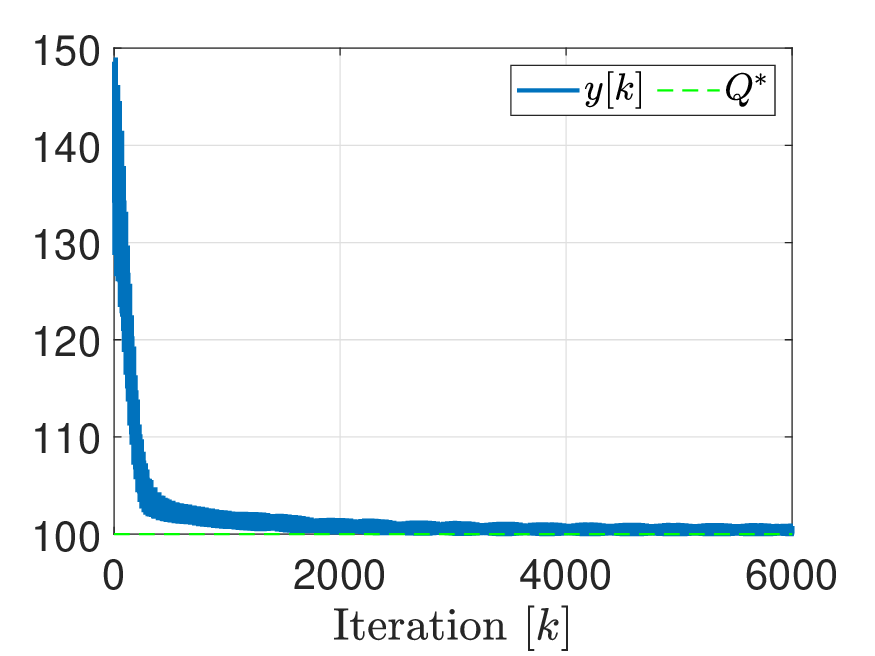}}
	\subfigure[Control input, \mbox{$u[k]$} \label{fig:U_siso}]{\includegraphics[width=8.8cm]{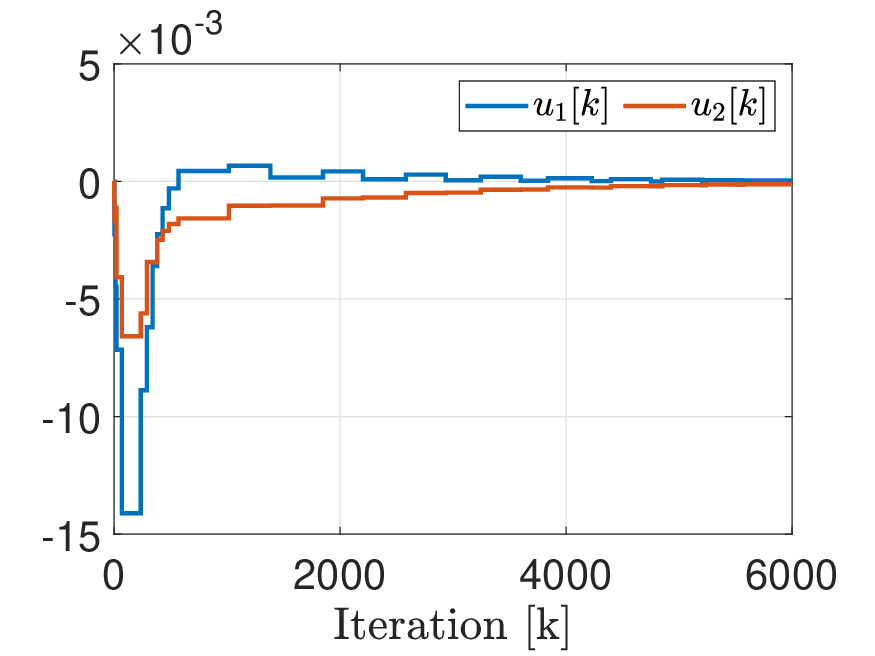}}
	~~
	\subfigure[Gradient estimate, \mbox{$\hat{G}[k]$}. \label{fig:hatG_siso}]{\includegraphics[width=8.8cm]{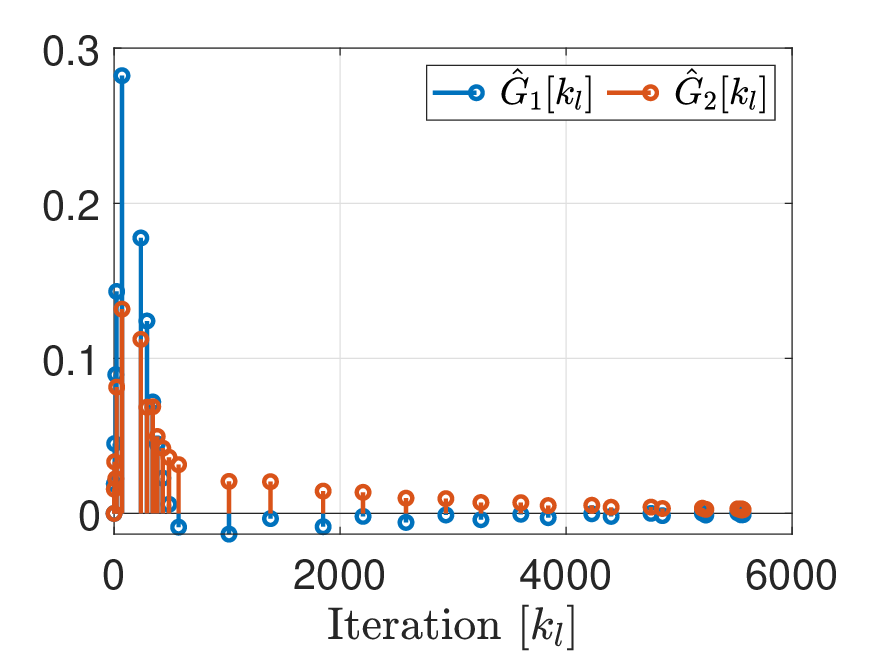}}   
	\caption{Simulations of the Event-triggered Discrete-time Multivariable  Extremum Seeking Control System. \label{fig:NESC_siso}}
    \begin{picture}(0,0)
        \put(-170,322){\includegraphics[width=3cm]{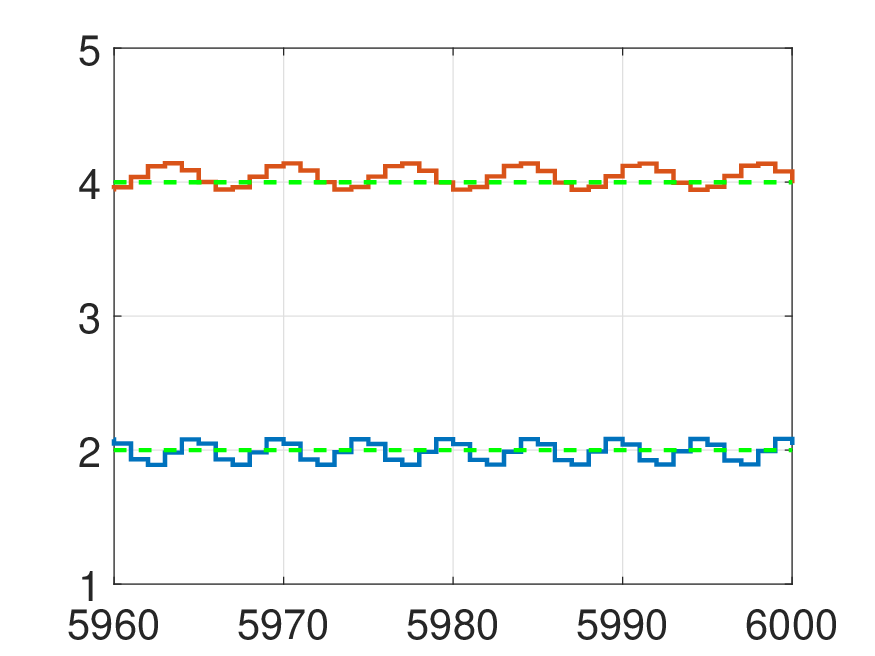}}
        \put(75,302){\includegraphics[width=4.5cm]{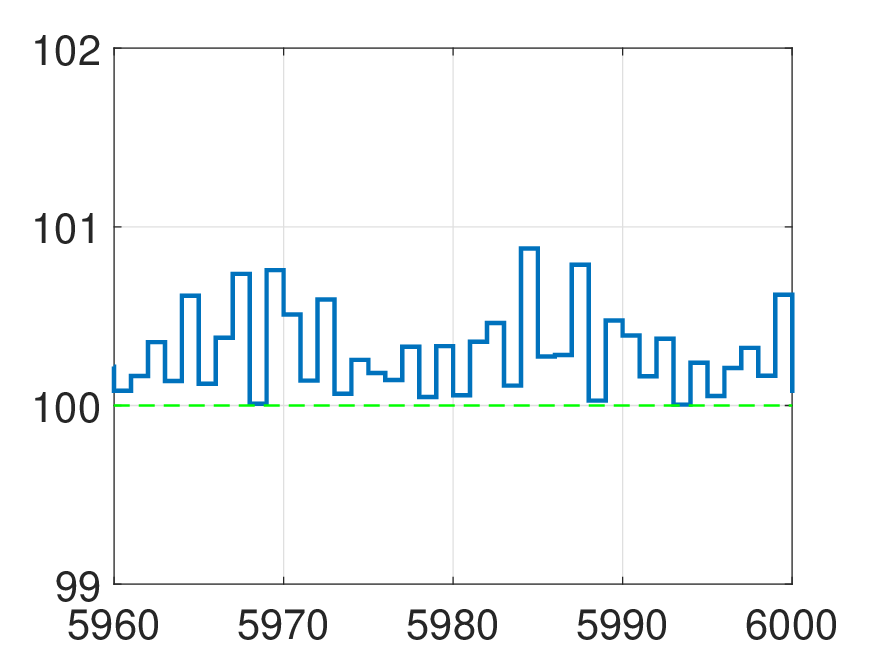}}
        \put(-185,85){\includegraphics[width=4cm]{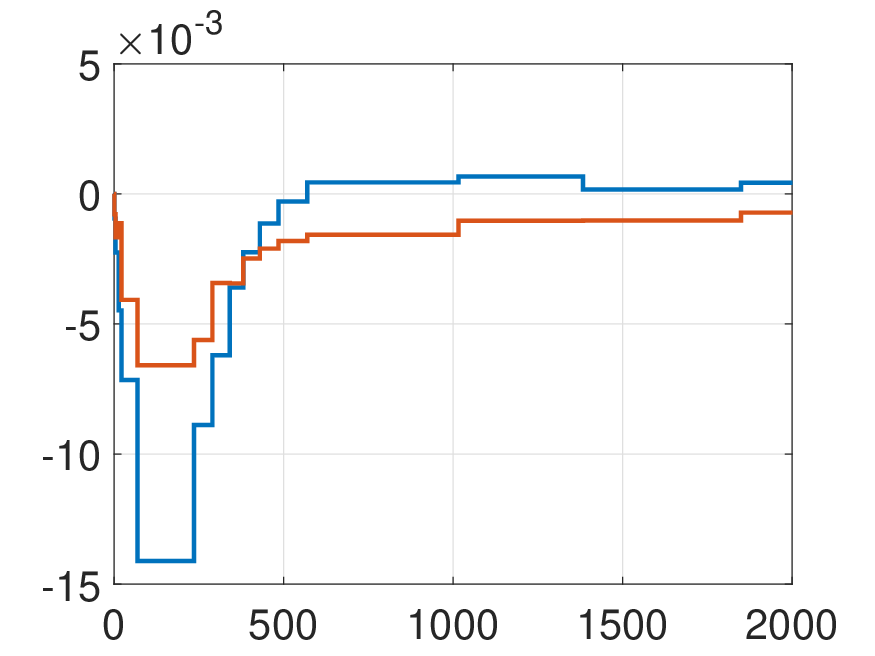}}
        \put(75,95){\includegraphics[width=4.5cm]{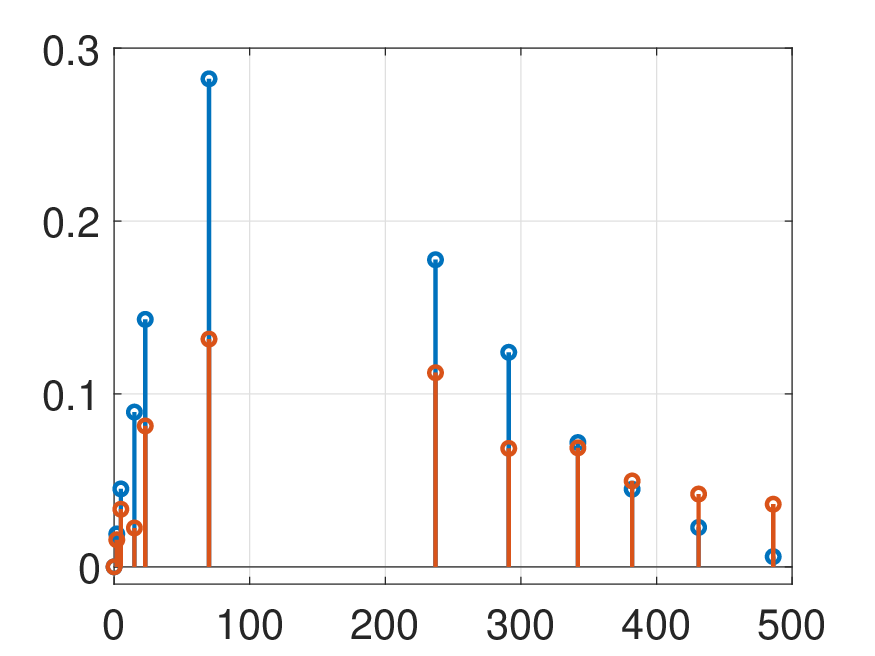}}
    \end{picture}
\end{figure*}
%


\section{Conclusions}

This paper developed a discrete-time event-triggered extremum seeking framework for the real-time optimization of multivariable static nonlinear maps, in which control updates are executed only when dictated by a state-dependent triggering condition. Differently from scalar formulations, the multivariable setting introduces coupling among channels and nontrivial sampled-data interactions that directly affect gradient recovery and stability properties under aperiodic updates. By combining discrete-time averaging arguments with Lyapunov-based analysis of the triggered dynamics, we established practical convergence of the closed-loop trajectories to a neighborhood of the unknown extremum and exponential stability of the associated average system. These results demonstrate that appropriately designed triggering rules preserve the essential optimization mechanisms of classical extremum seeking despite the loss of periodic sampling, while significantly reducing the number of input updates required for convergence.

The proposed framework provides a systematic bridge between extremum seeking and event-triggered control in discrete-time multivariable contexts, showing that resource-aware implementations can be rigorously reconciled with feedback optimization schemes that fundamentally rely on periodic excitation. Numerical results illustrated that comparable optimization performance can be achieved with substantially fewer actuation events, highlighting the suitability of the method for digital and networked control platforms. Future research directions include extensions to multi-agent and distributed extremum seeking architectures, robustness analysis in the presence of model uncertainties and delays \cite{TRoux:2022, ARC_2023, TDS_2025a, TDS_2025b,OKT:2017}, and experimental validation in embedded implementations, where the benefits of event-triggered discrete-time optimization are expected to be particularly impactful \cite{7035079, 7892977, 8827636, 10772652, 11219171, 11316674, 11316266, 11270207, 11316250, 11250918}.

\begin{small}

\end{small}

\appendices

\section*{Appendix}
\setcounter{equation}{0}
\renewcommand{\theequation}{\Alph{subsection}.\arabic{equation}}

\subsection{Averaging for Discrete-Time Systems \cite{BFS:1988,PPY:2004}}
\label{corpuschristi}

\textcolor{black}{
\begin{small}
In this section, we summarize the main results of discrete-time averaging theory. The proofs can be found in \cite{BFS:1988,PPY:2004} and are omitted here. 
Consider a difference equation of the form
\begin{equation}
x[k+1]=x[k]+\epsilon f(k,x[k],\epsilon)\,, \label{eq:BFS1988_eq1}
\end{equation}
where $x\in \mathbb{R}^n$, $k\in \mathbb{Z}_+$, $0<\epsilon\leq\epsilon_0$ and $f$ is piecewise continuous in $k$ with the limit
\begin{equation}
f_{\rm{av}}(x)=\lim_{\epsilon\rightarrow 0}\lim_{T\rightarrow\infty}
\frac{1}{T}
\sum_{k=s+1}^{s+T}
f(k,x,\epsilon) \,,  \label{eq:BFS1988_eq2}
\end{equation}
existing uniformly in $s$ and $\forall x\in B_r$, \textcolor{black}{a closed ball in $\mathbb{R}^n$ of radius $r$}. Assume that $f$ and $f_{\rm{av}}$ satisfy the following conditions ($\forall x\in B_r$, $0<\epsilon\leq\epsilon_0$ and $k\in \mathbb{Z}_+$):
\begin{itemize}
\item[(A1)]
$f(k,x,\epsilon)$ is Lipschitz in $x$, $$\|f(k,x_1,\epsilon)-f(k,x_2,\epsilon)\| \leq l_1\|x_1-x_2\|.$$
\item[(A2)] $f(k,x,\epsilon)$ is Lipschitz in $\epsilon$, linearly in $x$, $$\|f(k,x,\epsilon_1)-f(k,x,\epsilon_2)\| \leq l_2|\epsilon_1-\epsilon_2|\|x\|.$$
\item[(A3)]
$f_{\rm{av}}(x)$ is Lipschitz in $x$, $$\|f_{\rm{av}}(x_1)-f_{\rm{av}}(x_2)\| \leq l_{\rm{av}}\|x_1-x_2\|.$$
\item[(A4)] The average system of (\ref{eq:BFS1988_eq1}) is defined as
\begin{align}
x_{\rm{av}}[k+1]=x_{\rm{av}}[k] +\epsilon f_{\rm{av}}(x_{\rm{av}}[k])\,. \label{eq:BFS1988_eq3}
\end{align}
\end{itemize}
Then, we have the following theorems.
\begin{theorem}[Basic Averaging Theorem] \label{corpuschristi2}
Consider the original system (\ref{eq:BFS1988_eq1}) and the average system (\ref{eq:BFS1988_eq3}) satisfying the assumptions (A1)--(A4). For any given $T\in \mathbb{Z}_+$, further assume that the initial condition $x_0$ is sufficiently close to the equilibrium point $x^*$ of (\ref{eq:BFS1988_eq3}) 
so that $x_{\rm{av}}[k]\in B_r$, for some $r'<r$ and $k\in [0,[T/\epsilon]]$ (where $[T/\epsilon]$ denotes the largest integer $l$ such that $l\leq T/\epsilon$). Then, there is an $\epsilon_T$, $0<\epsilon_T\leq\epsilon_0$ and a class-$\mathcal{K}$ function $\Psi[\epsilon]$ (which is locally upper bounded by a linear function on its domain), so that
\begin{align}
    \|x[k]-x_{\rm{av}}[k]\| &\leq \Psi[\epsilon]\,b_T  \nonumber \\
    &\leq \mathcal{O}(\epsilon)\,, \label{argentinaeaustria}
\end{align}
for some $b_T>0$, on $k\in [0,[T/\epsilon]]$ and $0<\epsilon\leq\epsilon_T$.
\end{theorem}
\begin{theorem}[Exponential Stability Theorem to Residual Sets] \label{corpuschristi3}
Suppose that the conditions of Theorem~\ref{corpuschristi2} are satisfied. 
\textcolor{black}{In addition, if the equilibrium point $x^*$ of the average system (\ref{eq:BFS1988_eq3}) is exponentially stable, then there exists an $\epsilon_1$, $0<\epsilon_1\leq\epsilon_0$, such that (\ref{argentinaeaustria}) is valid for all $k \geq 0$ (infinite horizon) and  $x[k]$ in the original system (\ref{eq:BFS1988_eq1}) locally exponentially converges to an $\mathcal{O}(\epsilon)$ neighborhood of $x^*$ for all $0<\epsilon\leq\epsilon_1$ and $k\geq 0$.}
\end{theorem}
\end{small}
}

\subsection{Averaging Under Event-Triggered Switching}
\label{app:B}

This appendix provides the rigorous justification for the averaging step invoked in Sec.~III-E and used throughout Sec.~IV. It shows that although the vector field $f(k,x,h)$ in (\ref{eq:x_k+1_event}) is not globally $T$-periodic in $k$ once the event-triggering mechanism is active, the discrete-time averaging results of \cite{BFS:1988} (see Appendix~A) remain applicable on each inter-event interval, with the resulting $\mathcal{O}(h)$ estimate propagating uniformly over the infinite horizon.

\subsubsection{Preliminaries}
Recall from (\ref{eq:zeta_1}) that $\zeta[k] = \hat G[k_{l(k)}]$ is piecewise constant, equal to $\zeta_l := \hat G[k_l]$ for all $k \in [k_l, k_{l+1})$. Let $x[k] = (\hat G[k]^\top, \tilde\theta[k]^\top)^\top$ and let $x_{\rm{av}}[k]$ denote the solution of the average system (\ref{eq:hatG_av_k+1_1})--(\ref{eq:tildeTheta_av_k+1_1}), generated using its own event-triggering sequence $\{k_l^{\rm{av}}\}$ as in Definition~\ref{def:staticEvent_discrete_av}.

\begin{lemma}[Periodicity on the Frozen Extension]
\label{lem:A}
Fix $l \in \mathbb{N}$ and let $\zeta_l = \hat G[k_l]$. Define the auxiliary infinite-horizon system
\begin{equation}
x^{(l)}[k+1] = x^{(l)}[k] + h\, f\big(k, x^{(l)}[k], h; \zeta_l\big), \quad k \geq k_l,
\label{eq:aux_system}
\end{equation}
with $x^{(l)}[k_l] = x[k_l]$, where $f(\cdot,\cdot,\cdot;\zeta_l)$ denotes the field (\ref{eq:f1_20260321_1})--(\ref{eq:f2_20260321_1}) with $\zeta[k]$ held fixed at the constant value $\zeta_l$ for \emph{all} $k \geq k_l$. Then:
\begin{enumerate}
\item[(i)] $f(k,x,h;\zeta_l)$ is $T$-periodic in $k$, for every fixed $x \in \mathbb{R}^{2n}$, for all $k \geq k_l$.
\item[(ii)] Assumptions (A1)--(A4) of Appendix~A hold for \eqref{eq:aux_system} on any ball $B_r$, with Lipschitz constants $l_1, l_2, l_{\rm{av}}$ depending on $\zeta_l$ only through $\|\zeta_l\|$, $\|K\|$, and $\|H^*\|$.
\item[(iii)] Consequently, Theorem~2 (Appendix~A) applies directly to \eqref{eq:aux_system}, yielding
\begin{equation}
\big\|x^{(l)}[k] - x_{\rm{av}}^{(l)}[k]\big\| \leq \Psi(h)\, b_T = \mathcal{O}(h), \qquad \forall k \geq k_l,
\label{eq:lemA_bound}
\end{equation}
where $x_{\rm{av}}^{(l)}$ solves \eqref{eq:aux_system}'s average, and the bound holds from $k = k_l$ onward.
\end{enumerate}
\end{lemma}

\begin{proof}
(i) With $\zeta[k] \equiv \zeta_l$ constant, every $k$-dependence in $f_1$, given by (\ref{eq:f1_20260321_1}), enters exclusively through the periodic sequences $M[k]$, $\tilde M[k]$, $\Delta H^{\ast}[k]$, $\Delta \tilde{H}^{\ast}[k]$, $\delta[k]$, $\tilde{\delta}[k]$, defined in (\ref{eq:M_v1}), (\ref{eq:DeltaCalligraH}), (\ref{eq:Delta}) and (\ref{eq:hatG_20260313_4}), all of which are exactly $T$-periodic by construction and independent of $x$ and of the numerical value of $\zeta_l$. Hence $f(\cdot,x,\cdot;\zeta_l)$, viewed as a function of $k$ alone for fixed $x$, is $T$-periodic for every $k \geq k_l$, with no truncation.

(ii) Since $f$ is quadratic in $(x,\zeta_l)$ with coefficients bounded by $\|K\|,\|H^*\|,a_i$ (uniformly periodic and bounded), standard bounds on quadratic maps give, $\forall x_1,x_2 \in B_r$,
\begin{align}
\|f(k,x_1,h;\zeta_l) - f(k,x_2,h;\zeta_l)\| \leq l_1(\zeta_l) \|x_1-x_2\|\,,
\end{align}
with $l_1(\zeta_l)$ an explicit polynomial in $\|\zeta_l\|, \|K\|, \|H^*\|, r$, and analogously for $l_2, l_{\rm{av}}$. This establishes (A1)--(A3); (A4) is definitional.

(iii) Follows by direct application of Theorem~2 to \eqref{eq:aux_system}, whose hypotheses are satisfied globally in $k$ by (i)--(ii). The estimate \eqref{eq:lemA_bound} originates from the standard near-identity (stroboscopic) transformation underlying the proof of Theorem~2, which yields an estimate uniform in $k$ rather than one that improves only asymptotically in $k-k_l$.
\end{proof}

\begin{lemma}[Exact Restriction to the True Window]
\label{lem:B}
For all $k \in [k_l, k_{l+1})$, the true trajectory coincides with the auxiliary trajectory of Lemma~\ref{lem:A}:
\begin{equation}
x[k] = x^{(l)}[k], \qquad k \in [k_l, k_{l+1}).
\label{eq:lemB_identity}
\end{equation}
\end{lemma}

\begin{proof}
By (\ref{eq:zeta_1}) and (\ref{eq:zeta_2}), $\zeta[k] = \hat G[k_l] = \zeta_l$ identically for all $k \in [k_l,k_{l+1})$ in the true closed-loop system. Hence, on this range, the true recursion $x[k+1] = x[k] + h f(k,x[k],h;\zeta[k])$ coincides term-by-term with \eqref{eq:aux_system}. Since both systems share the initial condition $x[k_l]$, uniqueness of solutions to \eqref{eq:aux_system} gives \eqref{eq:lemB_identity}.
\end{proof}

Combining Lemmas~\ref{lem:A}--\ref{lem:B}:
\begin{equation}
\big\|x[k] - x_{\rm{av}}^{(l)}[k]\big\| \leq \mathcal{O}(h), \qquad k \in [k_l,k_{l+1}),
\label{eq:combined_AB}
\end{equation}
with the $\mathcal{O}(h)$ constant uniform in $l$ by Lemma~\ref{lem:A}(ii) and the a priori compactness of the trajectory (Sec.~IV-A).

\begin{lemma}[Reset Propagation]
\label{lem:C}
Let $x_{\rm{av}}[k]$ solve (\ref{eq:hatG_av_k+1_1}) and (\ref{eq:tildeTheta_av_k+1_1}) globally, switching to $\hat G_{\rm{av}}[k_l]$ at each $k_l$ per Definition~3. Then there exists $C>0$, independent of $l$, such that
\begin{equation}
\big\|x_{\rm{av}}^{(l)}[k] - x_{\rm{av}}[k]\big\| \leq C\, \big\|\hat G_{\rm{av}}[k_l] - \hat G[k_l]\big\|, \qquad k \in [k_l,k_{l+1}),
\label{eq:lemC_bound}
\end{equation}
and consequently
\begin{equation}
\big\|\hat G_{\rm{av}}[k_l] - \hat G[k_l]\big\| = \mathcal{O}(h), \qquad \forall l \in \mathbb{N}.
\label{eq:lemC_induction}
\end{equation}
\end{lemma}

\begin{proof}
Both $x_{\rm{av}}^{(l)}$ and $x_{\rm{av}}$ satisfy the linear recursion (\ref{eq:hatG_av_k+1_1}) on $[k_l,k_{l+1})$, differing only in initial condition: $x_{\rm{av}}^{(l)}$ freezes at $\hat G[k_l]$ (the true value), while $x_{\rm{av}}$ carries forward its own $\hat G_{\rm{av}}[k_l]$. Since $\|I_n - hH^*K\| \leq 1$ for $h$ sufficiently small (Assumption~2), the linear map is non-expansive, giving \eqref{eq:lemC_bound} with $C$ bounded uniformly in $l$.

For \eqref{eq:lemC_induction}, proceed by induction on $l$. Base case $l=0$: $\hat G_{\rm{av}}[0] = \hat G[0]$ exactly (identical initialization), so the left side is zero. Inductive step: assume $\|\hat G_{\rm{av}}[k_l]-\hat G[k_l]\| = \mathcal{O}(h)$. By \eqref{eq:combined_AB} and \eqref{eq:lemC_bound},
\begin{align}
\|\hat G[k_{l+1}] - \hat G_{\rm{av}}[k_{l+1}]\| &\leq \underbrace{\mathcal{O}(h)}_{\text{Lemma~\ref{lem:B}}} \nonumber \\
&\quad + \underbrace{C\rho^{(k_{l+1}-k_l)/2}\|\hat G_{\rm{av}}[k_l]-\hat G[k_l]\|}_{\text{contraction, Eq.~(\ref{eq:lyapfunc_10})}}\,,
\end{align}
where $\rho \in (0,1)$ is the contraction rate from (\ref{eq:lyapfunc_10}). Since $\rho^{(k_{l+1}-k_l)/2} < 1$ and the inductive hypothesis gives an $\mathcal{O}(h)$ term, the right-hand side remains $\mathcal{O}(h)$. Hence \eqref{eq:lemC_induction} holds for all $l$ by induction.
\end{proof}

\subsubsection{Main Result}
\begin{theorem}[Averaging Validity Under Event-Triggering]
\label{thm:mainB}
Under Assumptions~1--2 and the a priori compactness of the closed-loop trajectory established in Sec.~IV-A, the estimate
\begin{equation}
\|x[k] - x_{\rm{av}}[k]\| \leq \mathcal{O}(h), \qquad \forall k \geq 0,
\label{eq:main_result}
\end{equation}
holds, where $x_{\rm{av}}[k]$ is generated by (\ref{eq:hatG_av_k+1_1}) and (\ref{eq:tildeTheta_av_k+1_1}) using its own triggering sequence $\{k_l^{\rm{av}}\}$. Consequently, the estimates (\ref{eq:Gcloseness})--(\ref{eq:ecloseness}) used in the proof of Theorem~1 are valid as stated.
\end{theorem}

\begin{proof}
Fix $k \geq 0$ and let $l$ be such that $k \in [k_l,k_{l+1})$. By Lemma~\ref{lem:B}, $x[k] = x^{(l)}[k]$. By Lemma~\ref{lem:A}(iii) and Lemma~\ref{lem:C},
\begin{align}
\|x[k] - x_{\rm{av}}^{(l)}[k]\| &\leq \mathcal{O}(h)\,, \\
\|x_{\rm{av}}^{(l)}[k] - x_{\rm{av}}[k]\| &\leq C\|\hat G_{\rm{av}}[k_l]-\hat G[k_l]\| = \mathcal{O}(h),
\end{align}
so by the triangle inequality $\|x[k]-x_{\rm{av}}[k]\| \leq \mathcal{O}(h)$, uniformly over $l$ by Lemma~\ref{lem:C}'s induction, hence for all $k \geq 0$.
\end{proof}

Theorem~\ref{thm:mainB} addresses an apparent incompatibility between the short duration of the inter-event window $[k_l,k_{l+1})$ and the periodicity condition required by Theorem~2. The window need not contain a single period $T$ of the dither signals: since $k_{l+1}-k_l = \mathcal{O}(1)$ while $T/h \to \infty$ as $h \to 0$, it follows that $(k_{l+1}-k_l)h/T \to 0$.

This incompatibility is resolved in two steps. First, periodicity of the vector field, and the associated $\mathcal{O}(h)$ estimate, are established not on the true window but on the auxiliary, infinite-horizon extension \eqref{eq:aux_system} of the frozen field introduced in Lemma~\ref{lem:A}, which imposes no restriction on window length. The true trajectory inherits this estimate on $[k_l,k_{l+1})$ because it coincides \emph{exactly} with the auxiliary trajectory on that interval, as shown in Lemma~\ref{lem:B}, rather than because the window contains sufficiently many periods of its own.

Second, it must be verified that the $\mathcal{O}(h)$ error introduced at each reset does not accumulate over the resulting infinite sequence of switching intervals. This is established in Lemma~\ref{lem:C}, which employs the exponential contraction rate $\rho$ obtained from the Lyapunov analysis in Sec.~\ref{sec:ET-DT-gradient-ES_stability}.

\newpage
{
\renewcommand{\baselinestretch}{1}\selectfont
\begin{IEEEbiography}[{\includegraphics[width=1in,height=1.25in,clip,keepaspectratio]{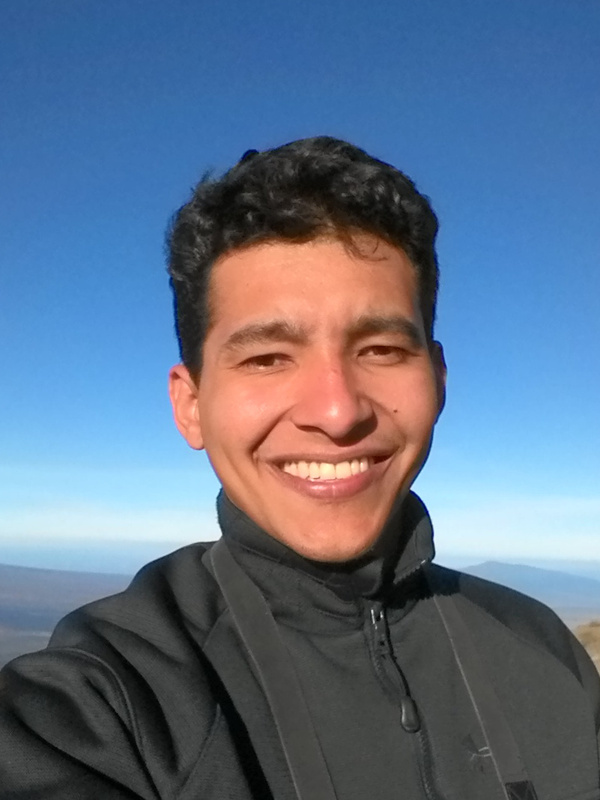}}]{Victor Hugo Pereira Rodrigues} received the B.Sc. and M.Sc. degrees in Electronics Engineering from the State University of Rio de Janeiro (UERJ) in 2016 and 2018, respectively. In 2022, he obtained the D.Sc. degree in Electrical Engineering from the Federal University of Rio de Janeiro. In 2020, he was awarded the Bolsa Nota 10 (Highest Rank Scholarship Prize), sponsored by the Brazilian agency FAPERJ. In 2024, he joined UERJ as an Associate Professor. 
His research interests include nonlinear control theory, adaptive and learning systems, hybrid systems, extremum seeking, and game theory.
\end{IEEEbiography} 
\begin{IEEEbiography}[{\includegraphics[width=1in,height=1.25in,clip,keepaspectratio]{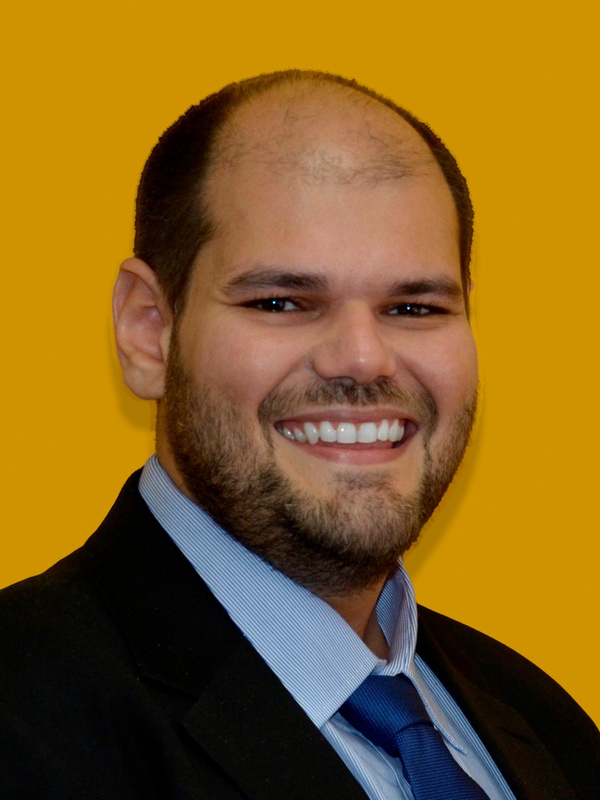}}]{Tiago Roux Oliveira} (Senior Member, IEEE) 
is currently Professor at State University of Rio de Janeiro and author of over 300 publications and the SIAM book Extremum Seeking through Delays and PDEs (2022). He serves as Associate Editor for several control journals (including IEEE Transactions on Automatic Control), chaired the IFAC TC on Adaptive and Learning Systems (2020–2026), and was elected President of the Brazilian Society of Automatics (2023–2025). His awards include 
the 2021 IEEE TCST Outstanding Paper Award.
\end{IEEEbiography}
\begin{IEEEbiography}[{\includegraphics[width=1in,height=1.25in,clip,keepaspectratio]{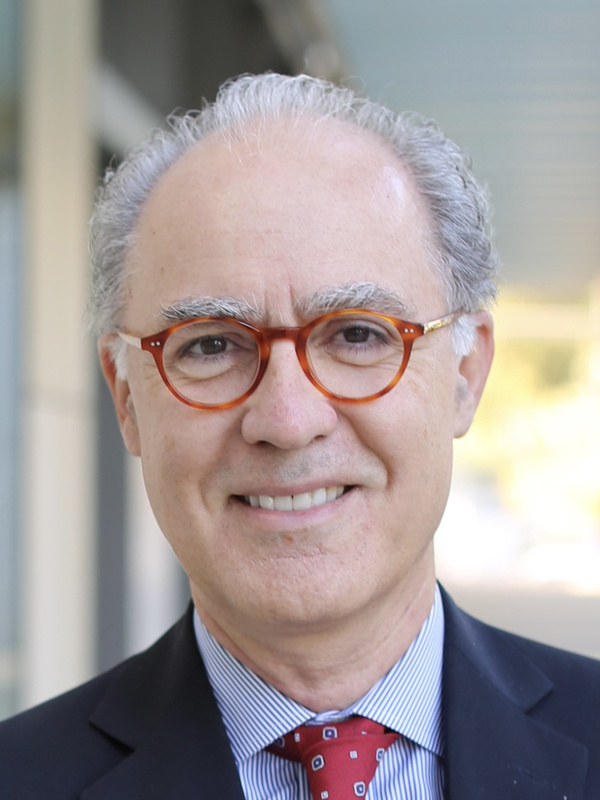}}]{Miroslav Krsti{\' c}} is Fellow of SIAM, IEEE, IFAC, ASME, AAAS, IET, AIAA (AF), and Serbian Academy of Sciences and Arts.
His awards include the SIAM Reid Prize, Bellman, Oldenburger, Ragazzini, Chestnut, Paynter, Nyquist Lecture, Bode Lecture, IFAC Nonlinear Control,
IFAC Ruth Curtain DPS, IFAC Adaptive and Learning Systems, Axelby, and Schuck (’96 and ’19. He has held chief or senior editorial positions for IEEE Transactions on Automatic Control, Systems \& Control Letters, and Automatica. Professor Krstic has co-authored 19 books on adaptive, nonlinear, and
stochastic control, extremum seeking, control of PDE systems including turbulent flows, and control of delay systems.
\end{IEEEbiography}
\begin{IEEEbiography}[{\includegraphics[width=1in,height=1.25in,clip,keepaspectratio]{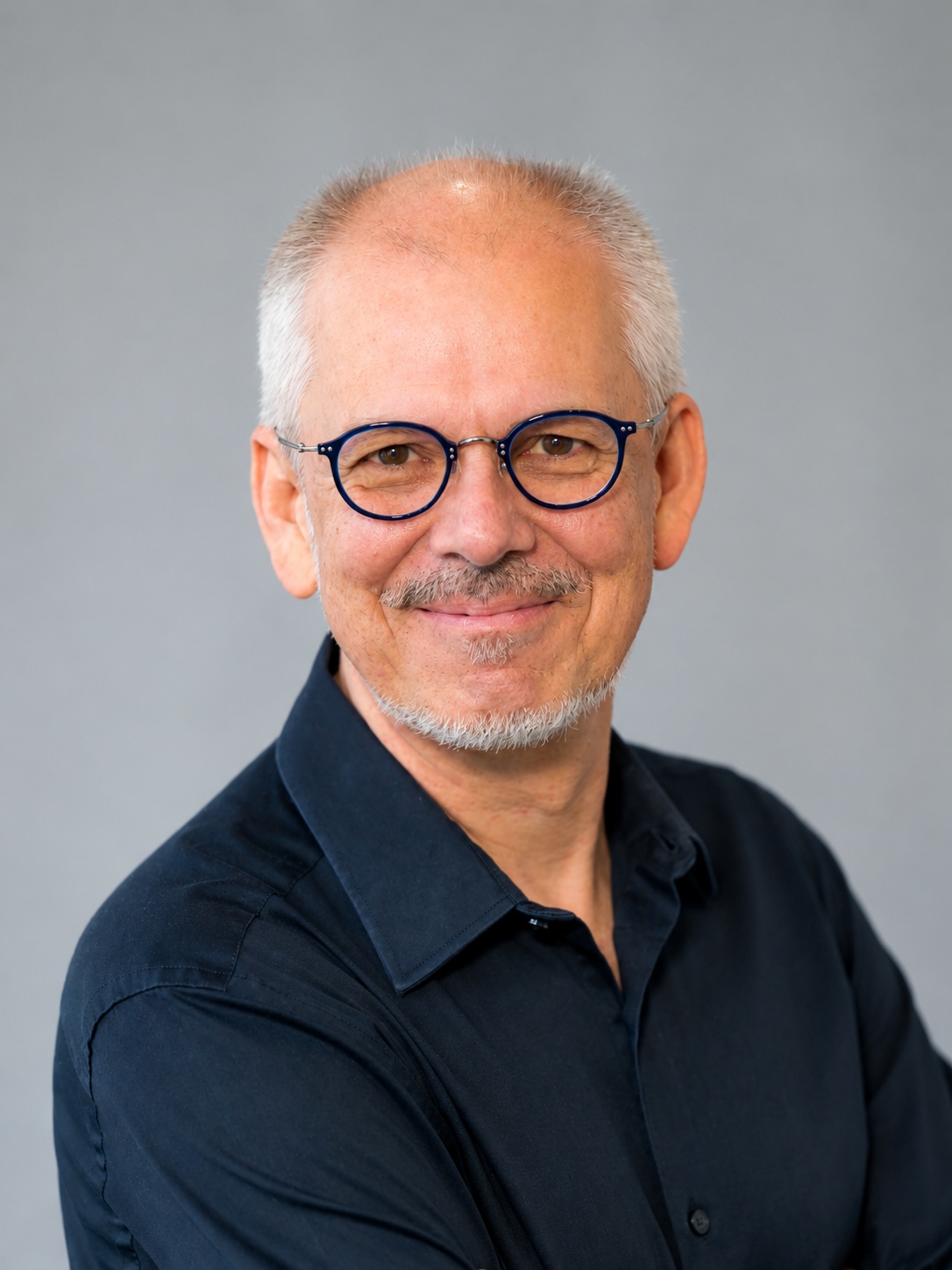}}]{Frank Allg{\" o}wer} is professor of mechanical engineering at the University of Stuttgart, Germany, and Director of the Institute for Systems Theory and Automatic Control (IST) there. Frank is active in serving the community in several roles: Among others he has been President of the IFAC for the years 2017–2020, Vice-president for Technical Activities of the IEEE Control Systems Society for 2013/14, and Editor of the journal Automatica from 2001 until 2015. From 2012 until 2020 Frank served in addition as Vice-president for the German Research Foundation, which is Germany’s most important research funding organization. His research interests include predictive control, data-based control, networked control, cooperative control, and nonlinear control with application to a wide range of fields including systems biology.
\end{IEEEbiography}
}

\end{document}